\documentclass[onefignum,onetabnum]{siamart171218}
\usepackage{lipsum}
\usepackage{amsfonts}
\usepackage{graphicx}
\usepackage{epstopdf}
\usepackage{algorithmic}
\ifpdf
  \DeclareGraphicsExtensions{.eps,.pdf,.png,.jpg}
\else
  \DeclareGraphicsExtensions{.eps}
\fi

\newsiamremark{remark}{Remark}
\newsiamremark{hypothesis}{Hypothesis}
\crefname{hypothesis}{Hypothesis}{Hypotheses}
\newsiamthm{claim}{Claim}

\headers{Incremental Methods for Weakly Convex Optimization}{X. Li, Z. Zhu, A. M.-C. So, J. D. Lee}

\title{ Incremental Methods for \\ Weakly Convex Optimization
}

\author{Xiao Li\thanks{School of Data Science, The Chinese University of Hong Kong, Shenzhen. 
		(\email{lixiao@cuhk.edu.cn}, \url{ https://sites.google.com/view/xli}).}
\and Zhihui Zhu\thanks{Department of Electrical and Computer Engineering, University of Denver. (\email{zhihui.zhu@du.edu}, \url{http://mysite.du.edu/\~zzhu61/}).}
\and Anthony Man-Cho So\thanks{Department of Systems Engineering and Engineering Management, The Chinese University of Hong Kong. 
	(\email{manchoso@se.cuhk.edu.hk}, \url{http://www.se.cuhk.edu.hk/\~manchoso}). }
\and Jason D Lee\thanks{ Department of  Electrical Engineering,   Princeton University. (\email{jasonlee@princeton.edu}, \url{https://jasondlee88.github.io/}).
}
}

\usepackage{amsopn}

\usepackage{array}
\usepackage{url}
\usepackage{amsmath,amssymb,amsbsy}
\usepackage{paralist}
\usepackage{xcolor}
\usepackage{color}
\usepackage{graphicx}
\graphicspath{{./figs/}}
\usepackage{algorithm}
\usepackage{algorithmic}
\usepackage{comment}
\usepackage{wrapfig}
\usepackage{caption}
\usepackage{subcaption}
\usepackage{bm}
\usepackage{framed}
\usepackage{fancyhdr}
\usepackage{cite}
\usepackage{cleveref}

\newtheorem{thm}{Theorem}%[section] %(If you want theorem numbered
\newtheorem{lem}{Lemma}%[section] %%    with section number.
\newtheorem{prop}{Proposition}%[section]
\newtheorem{defi}{Definition}%[section]

\newtheorem{assum}{Assumption}

\newcommand{\R}{\mathbb{R}}

\newcommand{\C}{\mathbb{C}}

\newcommand{\e}{\begin{equation}}
\newcommand{\ee}{\end{equation}}
\newcommand{\en}{\begin{equation*}}
\newcommand{\een}{\end{equation*}}
\newcommand{\eqn}{\begin{eqnarray}}
\newcommand{\eeqn}{\end{eqnarray}}
\newcommand{\bmat}{\begin{bmatrix}}
\newcommand{\emat}{\end{bmatrix}}
\newcommand{\btab}{\begin{tabular}}
\newcommand{\etab}{\end{tabular}}

\newcommand{\vct}[1]{\boldsymbol{#1}}
\newcommand{\mtx}[1]{\boldsymbol{#1}}
\newcommand{\T}{\mathrm{T}}

\newcommand{\rank}{\operatorname{rank}}

\newcommand{\dist}{\operatorname{dist}}

\DeclareMathOperator*{\argmin}{\text{argmin}}

\def \st {\operatorname*{subject\ to\ }}
\newcommand{\calA}{\mathcal{A}}

\newcommand{\calC}{\mathcal{C}}

\newcommand{\calO}{\mathcal{O}}
\newcommand{\calP}{\mathcal{P}}
\newcommand{\calQ}{\mathcal{Q}}

\newcommand{\calU}{\mathcal{U}}

\newcommand{\calX}{\mathcal{X}}

\newcommand{\calZ}{\mathcal{Z}}

\newcommand{\va}{\vct{a}}

\newcommand{\vs}{\vct{s}}

\newcommand{\vx}{\vct{x}}
\newcommand{\vy}{\vct{y}}
\newcommand{\vz}{\vct{z}}

\newcommand{\mA}{\mtx{A}}

\newcommand{\mR}{\mtx{R}}
\newcommand{\mS}{\mtx{S}}

\newcommand{\mU}{\mtx{U}}

\newcommand{\mX}{\mtx{X}}

\newcommand{\mId}{{\bm I}}

\graphicspath{{./figs/}}

\newlength{\imgwidth}
\newboolean{twoColVersion}
\setboolean{twoColVersion}{false}
\newcommand{\twoCol}[2]{\ifthenelse{\boolean{twoColVersion}} {#1} {#2} }

\usepackage[framemethod=tikz]{mdframed}

\begin{document}

\maketitle

% REQUIRED
\begin{abstract}
Incremental methods are widely utilized for solving finite-sum optimization problems in machine learning and signal processing. In this paper, we study a family of incremental methods---including incremental subgradient, incremental proximal point, and incremental prox-linear methods---for solving weakly convex optimization problems. Such a problem class covers many nonsmooth nonconvex instances that arise in engineering fields. We show that the three said incremental methods have an iteration complexity of $\calO(\varepsilon^{-4})$ for driving a natural stationarity measure to below $\varepsilon$. Moreover, we show that if the weakly convex function satisfies a sharpness condition, then all three incremental methods, when properly initialized and equipped with geometrically diminishing stepsizes, can achieve a local linear rate of convergence. Our work is the first to extend the convergence rate analysis of incremental methods from the nonsmooth convex regime to the weakly convex regime. Lastly, we conduct numerical experiments on the robust matrix sensing problem to illustrate the convergence performance of the three incremental methods.
\end{abstract}
% REQUIRED

% REQUIRED

\section{Introduction}
\label{submission}
Throughout this paper, we focus on the  \emph{finite-sum} optimization problem
\e
\begin{array}{c@{\quad}l}
\displaystyle \min_{\vx\in \R^n} & \displaystyle f(\vx)  = \frac{1}{m} \sum_{i = 1}^{m} f_i(\vx) \\
\st & \vx\in \calC,
\end{array}
\label{eq:problem}
\ee
where  each component function $f_i:\R^n\rightarrow \R$ is assumed to be \emph{weakly convex} and $\calC \subseteq \R^n$ is a nonempty  closed convex set.  Recall that the function $f$ is said to be $\tau$-weakly convex if $f(\cdot) + \tfrac{\tau}{2}\|\cdot\|_2^2$ is convex for some constant $\tau\ge0$~\cite{V83}.  
%Perhaps the most widely studied weakly convex optimization problem  is that the objective  function  has Lipschitz continuous gradient.  But more generally, a weakly convex function can be \emph{nonsmooth}.    A commonly seen class of weakly convex instances of problem \eqref{eq:problem} are of the  composite form 
%\e\label{eq:composit opt}
%f(\vx) = h(c(\vx))  = \frac{1}{m}  \sum_{i=1}^{m} \underbrace{  h_i(c_i(\vx)) }_{f_i(\vx)},
%\ee
%where $h_i: \R^d \rightarrow \R$ is a Lipschitz continuous convex function and $c_i: \R^n \rightarrow \R^d$ is a  smooth mapping with Lipschitz continuous Jacobian.  
%
%
%Recall that the objective function in~\eqref{eq:problem} has the \emph{finite-sum} structure $f = \tfrac{1}{m} \sum_{i=1}^m f_i$.
In  modern data processing applications, the number of components $m$ can be very large. Thus, it can be computationally expensive to utilize the full information (such as the subgradient) of $f$ in each update.  This observation is precisely one of the main motivations of incremental methods, which update the iterates using a \emph{single} component function $f_i$ rather than  \emph{all} the components of $f$.  It has been demonstrated by many experiments that depending on the value of $m$,  an incremental  method can significantly outperform  its nonincremental counterpart \cite{bertsekas2011incremental}.

Incremental methods are practical algorithms, which have a long history and are widely utilized in engineering fields.  Most notably,  incremental subgradient method and its randomly permuted version (also known as ``random shuffling'')  are  broadly employed in practice for training deep neural networks with a nonsmooth activation function; see e.g.,  \cite{bottou2009curiously,bottou2012stochastic,shamir2016without,gurbuzbalaban2019random,haochen2019random,bertsekas2018feature}. 
Despite the broad applications of incremental methods  to nonsmooth optimization problems, their theoretical properties are less explored.   The main prior results on  convergence rates of incremental methods when used to solve nonsmooth optimization problems are  based on \emph{convexity} assumptions, as stated in \cite[Section 4.1.3]{bertsekas2011incremental_a}. To the authors' knowledge, there is no convergence rate result for incremental  methods even when the objective function in \eqref{eq:problem} is  weakly convex (which can be nonsmooth and nonconvex).

Our interest in applying incremental methods to solve problem \eqref{eq:problem} stems from the fact that this problem appears widely in engineering fields such as signal processing and machine learning and \emph{almost no incremental methods have been proposed to solve it}. As a demonstration, we list below two motivating applications that are instances of problem \eqref{eq:problem}. The weak convexity of the objective functions in these applications follows from the fact that they all have the form 
\e\label{eq:composit opt}
f(\vx) = h(c(\vx))  = \frac{1}{m}  \sum_{i=1}^{m} \underbrace{  h_i(c_i(\vx)) }_{f_i(\vx)},
\ee
where $h_i : \mathbb{R}^d \rightarrow \mathbb{R}$ is a Lipschitz continuous convex function and $c_i : \mathbb{R}^n \rightarrow \mathbb{R}^d$ is a smooth mapping with Lipschitz continuous Jacobian (see, e.g., \cite[Lemma 4.2]{drusvyatskiy2018efficiency}).
% and will show in the experimental section that our incremental methods have  state-of-the-art performances when applied to solve these applications. 

\subsection{Motivating applications}\label{sec:mov} 
%Here, we list three nonsmooth nonconvex examples arising in signal processing and machine learning.  All the problems have their objective function in the  composite form  \eqref{eq:composit opt}  and hence give rise to weakly convex optimization. 
\paragraph{Application 1: Robust  Matrix Sensing \cite{li2018nonconvex}}  Low-rank matrices are ubiquitous in computer vision, signal processing, machine learning, and data science applications. One fundamental computational task is to recover a positive semidefinite (PSD) matrix $\mX^\star \in\R^{n\times n}$ with $\rank(\mX^\star) = r \ll n$ from a  number of corrupted linear measurements 
\e
\vy = \calA(\mX^\star) + \vs^\star,
\label{eq:rms model}\ee
where $\calA:\R^{n\times n} \rightarrow \R^m$ is a linear measurement operator consisting of a set of sensing matrices $\mA_1,\ldots,\mA_m \in \R^{n\times n}$  and $\vs^\star\in \R^{m}$ is a sparse  vector with arbitrary nonzero entries (i.e., outliers).  The work  \cite{li2018nonconvex} proposes to recover the low-rank matrix $\mX^\star$  by using a factored representation of the matrix variable \cite{burer2003nonlinear} (i.e., $\mX = \mU\mU^\top$ with $\mU\in\R^{n\times r}$) and employing a $\ell_1$-loss function to robustify the solution against outliers. This leads to the optimization formulation
\e
\min_{\mU\in \R^{n\times r}} \ \frac{1}{m}\|\vy - \calA(\mU\mU^\top) \|_1 = \frac{1}{m} \sum_{i=1}^{m}| y_i - \langle \mA_i,  \mU\mU^\top\rangle  |.
\label{eq:rms factorization}
\ee
%
%\cite[Proposition 3]{li2018nonconvex} shows that the components of the objective function in \eqref{eq:rms factorization} is weakly convex and thus the above formulation is an instance of \eqref{eq:problem}. 

%Furthermore, it is shown in \cite{li2018nonconvex}, under certain statistical assumptions ($\calA$ has \emph{i.i.d.} Gaussian ensembles and the fraction of outliers in $\vs^\star$ is less than $\frac{1}{2}$ and $m \geq \calO(nr)$), $\calU = \{\mU^\star \mR: \mR\in \R^{r\times r},\mR\mR^T = \mId\} $ is exactly the set of global minimizers to \eqref{eq:rms model} and  the objective function in \eqref{eq:rms model} is sharp with parameter $\alpha = c \cdot \sigma_r(\mX^\star)$ where $c>0$ is a constant depending on the fraction of outliers in $\vs^\star$. 

\paragraph{Application 2: (Real-Valued) Robust Phase Retrieval \cite{duchi2017solving}} An important problem arising in physics, imaging science, and signal processing is phase retrieval, which aims to recover a  signal $\vx^\star \in \C^n$ from its amplitude measurements. For the purpose of illustration, let us consider a real-valued version of the problem. Suppose that the measurements are given by
\e \label{eq:rpr measure}
\vy = |\mA\vx^\star|^2 + \vs^\star,
\ee
where the operator $|\cdot|^2$ in \eqref{eq:rpr measure} is applied component-wise to its arguments. Here,  $\mA\in \R^{m\times n}$ is the measurement matrix and $\vs^\star\in \R^{m}$ is a sparse  vector with arbitrary nonzero entries (i.e., outliers). The work \cite{duchi2017solving} considers the following formulation for recovering both the sign and magnitude of $\vx^\star$:
\[ %\e \label{eq:rpr}
\min_{\vx\in\R^n} \ \frac{1}{m} \left\| \vy -  |\mA\vx|^2  \right\|_{1} = \frac{1}{m} \sum_{i=1}^{m}\left| y_i - |\langle \va_i, \vx\rangle|^2 \right|.
\] %\ee

\subsection{Main contributions}
In this paper, we study a family of incremental methods---including \emph{incremental subgradient}, \emph{incremental proximal point}, and \emph{incremental prox-linear methods}---for solving problem \eqref{eq:problem}. Among them, the incremental prox-linear method is new to our knowledge. We develop a unified framework for analyzing the convergence rates of these methods. In particular, we show that the three incremental methods mentioned above drive a surrogate stationarity measure to zero at a rate of $\calO(k^{-1/4})$, where $k$ is the iteration index (see \Cref{thm:global convergence rate}). In addition, we show that if problem \eqref{eq:problem} possesses the so-called sharpness property (see \Cref{def:sharpness}), then the three incremental methods with properly designed geometrically diminishing stepsizes and a good initialization will converge to the set of weak sharp minima at a \emph{linear} rate (see \Cref{thm:local linear convergence}).  Our work is the first to extend the convergence rate analysis of incremental methods from the nonsmooth convex regime to the weakly convex regime, which covers a large class of nonsmooth nonconvex problems. % (see \Cref{sec:mov}).

Unlike the incremental aggregated gradient methods (see e.g., \cite{blatt2007convergent,gurbuzbalaban2017convergence,mokhtari2018surpassing}), which has a linear rate of convergence for smooth strongly convex optimization problems, we do not incorporate any aggregation techniques in our incremental methods. Moreover, it should be noted that the full subgradient method also requires geometrically diminishing stepsizes to guarantee linear convergence, and it is usually outperformed by its incremental counterpart numerically (see \Cref{sec:experiments} for an illustration). This further demonstrates the promise of the incremental subgradient method.
%The insight for showing the linear convergence result is the \emph{cyclical} updating of incremental methods,   which utilizes all the components after $m$ inner iterations.

Our linear convergence result generalizes the original ones in \cite{nedic2001incremental,nedic2001convergence,bertsekas2011incremental}, which concern the incremental subgradient and proximal point methods for nonsmooth \emph{convex} optimization problems. To obtain the global sublinear convergence result, we adopt the surrogate stationarity measure for weakly convex minimization problems from \cite{davis2018stochastic,drusvyatskiy2018efficiency}.

\subsection{Related works}
%We review the history of incremental subgradient method, incremental proximal point algorithm and incremental prox-linear algorithm individually.
%The nonincremental subgradient method, proximal point method and prox-linear algorithm are substantially studied for both convex and nonconvex optimization.
\paragraph{Incremental gradient method}
A popular algorithm for solving problem \eqref{eq:problem} when the components are smooth is the incremental gradient method. Such a method has a long tradition and is extensively studied. 
The starting work dates back to \cite{widrow1960adaptive} for solving linear least-squares problems. Then, various works \cite{zhi1994analysis,mangasariany1994serial,grippo2000convergent,tseng1998incremental,solodov1998incremental}
study the convergence of incremental gradient method  with different stepsize schemes.  Most of them establish asymptotic convergence of the method without providing an explicit convergence rate. For instance, Solodov \cite{solodov1998incremental} showed that when applied to a smooth nonconvex optimization problem, every limit point of the sequence of iterates generated by the incremental gradient method with a constant stepsize bounded away from zero is an approximate stationary point.   
%The work \cite{luo1991convergence} considers the strongly convex linear least squares optimization and shows that the iterates generated by incremental gradient method will converge to the optimal solution at a rate of $\calO(1/k)$.  
%As an extension to \cite{luo1991convergence},  a much 
A more recent work \cite{gurbuzbalaban2015convergence} shows that if the objective function $f$ is strongly convex and twice continuously differentiable and the iterates are uniformly bounded, then the incremental gradient method with stepsizes diminishing at the rate of $\calO(1/k)$ will drive the distances between the iterates and the optimal solution to zero at an asymptotic rate of $\calO(1/k)$. The analysis in \cite{gurbuzbalaban2015convergence} was later adopted by several works \cite{haochen2019random,nagaraj2019sgd,gurbuzbalaban2019random} for showing the convergence of random shuffling (a randomly permuted version of incremental gradient method) when applied to smooth  strongly convex optimization problems. A popular variant named incremental aggregated gradient method \cite{blatt2007convergent,gurbuzbalaban2017convergence,mokhtari2018surpassing}---which, in each update,  evaluates  the gradient of  a single component function while keeping a memory of the most recent gradients of all the other components  to approximate the full gradient---is shown to converge linearly for  smooth  strongly convex minimization using a constant stepsize. 

\paragraph{Incremental subgradient method} The incremental subgradient method  is widely used to tackle  problem \eqref{eq:problem} when the components are nonsmooth. In the nonsmooth setting, almost all the existing convergence rate results require convexity of the components; see the comments in \cite[Section 4.1.3]{bertsekas2011incremental_a}.   As reviewed in \cite{bertsekas2011incremental_a}, the incremental subgradient method was first proposed in \cite{kibardin79}.  After that, 
Nedi\'c and Bertsekas~\cite{nedic2001incremental,nedic2001convergence}  provided  convergence results for the incremental subgradient method using several stepsize rules  when the components are convex. In particular, they proved that the algorithm converges at a rate of $\calO(1/\sqrt{k})$ in terms of the function suboptimality gap $f(\vx_k) - f^\star$ when a constant stepsize or  Polyak's dynamic stepsizes is used. Furthermore, under an additional sharpness property (see \Cref{def:sharpness}), they proved that the algorithm with Polyak's dynamic stepsizes will drive  the distances between the iterates and the optimal solution set to zero at a linear rate.
Later, the works \cite{kiwiel2004convergence,neto2010incremental,nedic2010effect} establish  convergence results   for the incremental $\varepsilon$-subgradient method, in which the exact subgradient oracle is not available. Among these,  the work \cite{nedic2010effect} also considers the effect of deterministic noise in the update.
%  \cite{rabbat2004distributed}  applied incremental subgradient method to sensor network optimization. 
There are also many other works studying the incremental subgradient method for the setting where the components are convex in the context of distributed optimization, sensor network optimization, etc; see, e.g.,  \cite{nedich2001,rabbat2004distributed,rabbat2005quantized,joh2010,ram2009,nedic2009distributed,wang2015incremental,iiduka2014acceleration}.  By contrast, our global iteration complexity and  linear convergence results apply to weakly convex minimization, in which the objective function can be nonsmooth and nonconvex. In addition, our linear convergence result builds on a more practical geometrically diminishing stepsize rule (due to Shor)  than the Polyak's  stepsize rule, which requires the knowledge of the optimal function value $f^\star$.

There are also  works  considering the incremental subgradient method for nonsmooth nonconvex minimization. In particular,  Solodov \cite{solodov1998error} proposed perturbed subgradient-type methods that include the incremental subgradient method as a special case and showed that the iterates generated by these methods with  diminishing stepsizes  will converge to the set of first-order stationary points of a  Lipschitz continuous and regular function (in the sense of Clarke). Despite the generality of this work, the author did not report any convergence rate result even in the case where the objective function is restricted to be convex.  It is also worth mentioning that  if the problem is convex and possesses the sharpness property, then the work \cite{solodov1998error} shows that the iterates generated by the perturbed incremental subgradient method will  converge to the set of exact global minimizers as long as the magnitude of the perturbation is smaller than the sharpness parameter.   A more recent work \cite{hu2019incremental} studies the case where the components are quasi-convex. Under the stringent assumption that all the components have a common optimal solution, it is shown that the incremental subgradient method asymptotically converges to the  optimal solution set, but there is no rate guarantee. 
By contrast, we establish an explicit rate of convergence of the incremental subgradient method when applied to  weakly convex minimization problems.

\paragraph{Incremental proximal point method}
Bertsekas \cite{bertsekas2011incremental} proposed the incremental proximal point method for solving problem \eqref{eq:problem} with convex components.  The method includes the incremental gradient, subgradient, proximal gradient,  and proximal subgradient methods as special cases. By utilizing proof techniques similar to those in \cite{nedic2001incremental,nedic2001convergence}, the author proved an $\calO(1/\sqrt{k})$ convergence rate in terms of the function suboptimality gap when a constant stepsize is used. Such a convergence rate coincides with that of the incremental subgradient method \cite{nedic2001incremental,nedic2001convergence}.  The work \cite{wang2016st} considers a slightly more general case where the constraint is an intersection of a large number of simple convex constraints and proposes an incremental projection-proximal method for solving it. It is shown in  \cite{wang2016st} that the method has an $\calO(1/\sqrt{k})$ convergence rate.  To our knowledge, the convergence  of the incremental proximal point method has only been studied in the convex setting. By contrast, we consider the incremental proximal point method  for weakly convex minimization, where the problem can be nonsmooth and nonconvex.  In addition, we elucidate the role of sharpness in the  linear convergence analysis of the incremental proximal point method, which is not addressed in \cite{bertsekas2011incremental,wang2016st}.

%To our knowledge, the convergence performance of incremental proximal point algorithm has not yet been studied for nonconvex problems.  In terms of incremental prox-linear algorithm, to our knowledge it has not yet been utilized or analyzed in the literature, except its stochastic counterpart which was proposed and analyzed very recently in \cite{duchi2018stochastic,davis2018stochastic}.

\paragraph{Stochastic methods}
In the past decade, stochastic methods have been extensively studied in optimization and machine learning. There is a substantial literature on stochastic variants of subgradient and proximal-type methods.  Actually, the works \cite{nedic2001incremental,nedic2001convergence,bertsekas2011incremental} mentioned earlier also discussed stochastic variants of incremental subgradient and proximal point methods, which randomly select a component function from $\{f_1, \ldots, f_m\}$ in each update rather than sweeping all the components in a cyclic order. Recently,  several works \cite{duchi2018stochastic,davis2018stochastic,davis2019stochastic,davis2019prox} have proposed different stochastic methods for solving weakly convex minimization problems. 
In particular, it is shown in \cite{davis2018stochastic} that a family of stochastic methods will drive a certain surrogate stationarity measure to zero at a rate of $\calO(k^{-1/4})$. Central to the analysis in \cite{davis2018stochastic} is the observation that the Moreau envelope of a weakly convex function can be regarded  as an approximate Lyapunov function and  the algorithm dynamics can drive the gradient of the Moreau envelope to zero.  In this paper, we adopt such a surrogate stationarity measure for analyzing the global iteration complexity of our incremental methods. In addition, the work \cite{davis2019stochastic}, which appeared on arXiv one week before our preliminary technical report \cite{li2019incremental}, establishes the linear convergence of several stochastic methods when applied to sharp weakly convex minimization problems. However, the stochastic algorithms proposed in \cite{davis2019stochastic} incorporate a restarting strategy, which results in a computationally heavy inner loop. When the problem is large-scale and/or high-dimensional, those algorithms can be very inefficient. By contrast, when solving sharp weakly convex optimization problems, our linearly convergent incremental methods do not need any restarting scheme, which is much more efficient in practice. It is worth pointing out that the necessity of investigating incremental methods when there are already results for their stochastic counterparts can be seen from three aspects: 1) The analysis of stochastic methods  relies heavily on the stochastic nature of the component function selection rules and  does not carry over to incremental methods. 2) There are applications---such as source localization problems in sensor networks and distributed empirical risk minimization \cite{gurbuzbalaban2017convergence}---where random sampling required by the stochastic methods may not be possible since the access to the components are predetermined in a deterministic order by the problem's physical nature. 3) The analysis of cyclic updating order in incremental methods can be  crucial for understanding more general component function selection rules such as the above mentioned random shuffling algorithm; see, e.g., \cite{gurbuzbalaban2019random}. 
%Stochastic methods are similar to the incremental ones in the sense that only one component function is used for update at each iteration. The essential difference between these two types of methods lies in the component function selection rule. 

For the sake of clarity, we list several representative results on incremental methods and compare them with our results in \Cref{table:increemental methods}.

\setlength{\arrayrulewidth}{0.4mm}
\begin{table*}[t]\caption{Summary of prior arts. The algorithm used in \cite{nedic2001convergence} is the incremental subgradient method, while that used in \cite{bertsekas2011incremental} is the incremental proximal point method. The results in this paper cover the incremental subgradient, proximal point, and  prox-linear methods. Here, $k$ denotes the iteration index and $T$ denotes the total number of iterations. In  Polyak's rule, $f^\star$  denotes the optimal function value and  $\widetilde \nabla f (\vx_{k}) \in  \partial f(\vx_{k}) $ denotes a nonzero subgradient. Furthermore, $\Theta(\vx_k)$ denotes a surrogate stationarity measure of a weakly convex function $f$ (see \eqref{eq:surrogate stationarity}), $\dist(\vx,\calZ):=\inf_{\vz \in \calZ} \|\vx-\vz\|_2$ denotes the Euclidean distance between the point $\vx$ and the non-empty closed set $\calZ$, and $\calX$ denotes the set of weak sharp minima of problem~\eqref{eq:problem} (see \Cref{def:sharpness}). We hide numerical constants in the big-oh notation for a cleaner display.}\label{table:increemental methods}
	\begin{center}
		\small{
			\setlength\tabcolsep{2.3pt}
			\setlength\extrarowheight{2.3pt}
			\begin{tabular}{c|c|c|c|c}
				\hline
				Paper&Assumptions&Stepsize&Complexity&Stationarity Measure\\
				\hline  \hline
%				\cite{gurbuzbalaban2015convergence}  & \btab{c}  $f$ strongly convex\\ $f_i$ convex smooth \\ $f_i$ Lipschitz   \\ $ \nabla f_i$ Lipschitz    \etab   & \btab{c}  Polynomially \\ diminishing\\ $\calO\left(\frac{1}{k} \right)$ \etab &$T= \calO\left( \frac{1}{ \varepsilon}\right)$&$\dist\left(\vx_T, \calX\right) \leq \varepsilon$\\
%				\hline
				\cite{nedic2001convergence} & \btab{c} $f_i$ convex \\
				$f_i$ Lipschitz   \etab &  \btab{c}  Constant \\ $\calO\left(\frac{1}{\sqrt{T+1}} \right)$ \etab & $T= \calO\left( \frac{1}{\varepsilon^2}\right)$ &$\min\limits_{0 \leq k \leq T} f(\vx_{k}) -  f^\star \leq \varepsilon$ \\
				\hline
				\cite{nedic2001convergence} &\btab{c} $f_i$ convex \\
				$f_i$ Lipschitz \\ $f$ sharp   \etab &  \btab{c}  Polyak's rule\\ $\frac{f(\vx_k) - f^\star}{ \left\|\widetilde \nabla f (\vx_{k}) \right\|^2_2 } $ \etab & $T= \calO\left( \log\left( \frac{1}{\varepsilon} \right)\right) $& $\dist\left(\vx_T, \calX\right) \leq \varepsilon$\\
				\hline
				\cite{bertsekas2011incremental}&\btab{c} $f_i$ convex \\
				$f_i$ Lipschitz   \etab& \btab{c}  Constant \\ $\calO\left(\frac{1}{\sqrt{T+1}} \right)$ \etab&$T = \calO\left( \frac{1}{\varepsilon^2}\right)$ &$\min\limits_{0 \leq k \leq T} f(\vx_{k}) -  f^\star \leq \varepsilon$ \\
				\hline
				\btab{c}  This paper  \\ \Cref{thm:global convergence rate} \etab &\btab{c} $f_i$ weakly convex \\
				$f_i$ Lipschitz  \etab &\btab{c}  Constant \\ $\calO\left(\frac{1}{\sqrt{T+1}}\right) $ \& \\ diminishing \\ $\calO\left( \frac{1}{\sqrt{k}}\right)$ \etab& $T=  \calO\left( \frac{1}{\varepsilon^4}\right)$ &$\min\limits_{0 \leq k \leq T} \Theta(\vx_k) \leq \varepsilon$ \\
				\hline
				\btab{c}  This paper  \\ \Cref{thm:local linear convergence}\etab &\btab{c} $f_i$ weakly convex \\
				$f_i$ Lipschitz \\ $f$ sharp \\ Good initialization\etab &\btab{c}  Geometrically \\ diminishing \\ $\rho^k\cdot \mu_0  $ \etab& $T= \calO\left( \log\left( \frac{1}{\varepsilon} \right)\right) $ &$\dist\left(\vx_T, \calX\right) \leq \varepsilon$ \\
				\hline
			\end{tabular}
		}
	\end{center}
\end{table*}

%\subsection{Notation}

\section{Preliminaries}
In this section, we review several elements of nonsmooth analysis and specialize them to weakly convex functions, including the subdifferential, subdifferential calculus for the composite form \eqref{eq:composit opt}, and a subgradient inequality. We then  present a family of incremental methods for solving problem \eqref{eq:problem}.

\subsection{Subdifferential,  first-order optimality condition, and a subgradient inequality}
Let $f:\R^n\rightarrow \R$ be a $\tau$-weakly convex function for some parameter $\tau\geq 0$. By \cite[Proposition 4.3]{V83}, there is a convex function $\phi:\R^n\rightarrow \R$ such that $f(\vx) = \phi(\vx) - \frac{\tau}{2}\|\vx\|^2_2$  for any $\vx\in\R^n$. According to \cite[Proposition 4.6]{V83}, we have 
\e \label{eq:subdifferential}
   \partial f(\vx) = \partial \phi (\vx) - \tau \vx,
\ee
where $\partial \phi (\vx)$ is the usual convex subdifferential of $\phi$ at $\bm{x}$. Thus, the above subdifferential of the weakly convex function $f$ is well defined. 

It is not always immediate how to explicitly calculate the subdifferential of $f$ from \eqref{eq:subdifferential}, as it may not be easy to obtain the  convex function $\phi$ associated with $f$.  Nevertheless, when $f$ takes the composite form \eqref{eq:composit opt}, we have 
\[ %\e\label{eq:sub calculus}
    \partial f(\vx) = \nabla c(\vx)^\top \partial h(c(\vx))
\] %\ee
by \cite[Theorem 10.6]{rockafellar2009variational} and \cite[Proposition 4.5]{V83}. An element  $\widetilde \nabla f(\vx) \in \partial f(\vx)$ is called a subgradient of $f$ at $\vx$. Using the definition of the subdifferential of $f$ in \eqref{eq:subdifferential}, we have the following equivalent characterization of the $\tau$-weak convexity of $f$: For  all $\vx, \vy \in \R^n$,
\e
f(\vy) \geq  f(\vx) +   \langle \widetilde \nabla f(\vx), \vy-\vx\rangle -    \frac{\tau}{2} \|\vy - \vx\|^2, \quad \forall \ \widetilde \nabla f(\vx)  \in \partial f(\vx);
\label{eq:weak convexity}
\ee
see, e.g., \cite[Proposition 4.8]{V83}. 

We call $\vx\in \calC$  a \emph{stationary point} of problem \eqref{eq:problem} if it satisfies
\e \label{eq:first order optimality}
  \mathbf 0 \in \partial f(\vx) + N_{\calC} (\vx),
\ee
where  $N_{\calC}(\vx) = \{\vs\in \R^n: \langle \vs, \vy - \vx \rangle \leq 0, \ \forall \vy \in \calC \}$ is the normal cone to the convex set $\calC$ at $\vx$.

\subsection{The incremental methods}\label{sec:incre methods}
Let us now introduce a family of incremental methods---namely the \emph{incremental subgradient}, \emph{incremental proximal point}, and \emph{incremental prox-linear methods}.
%The last one is tailored for minimization of the composition of a Lipschitz continuous convex function and a smooth mapping which is a subclass of the more general weakly convex function $f$ in \eqref{eq:problem}.
At each iteration, the incremental methods update $\vx_k$ to $\vx_{k+1}$ through $m$ sequential steps by utilizing the components  $\{f_1, \ldots, f_m\}$ \emph{sequentially}. In each step, only \emph{one} component function $f_i\in \{f_1,\ldots, f_m\}$ is selected for updating. To be more specific, at the $(k+1)$-st iteration,
\begin{framed}
incremental methods start with
$\vx_{k,0} = \vx_{k}$
and then update  $\vx_{k,i} $ using $f_i$  for  $i = 1,\ldots, m$, giving
$\vx_{k+1} = \vx_{k,m}.$
\end{framed}
Let $\mu_{k}$ be the  stepsize at the $(k+1)$-st iteration. The three incremental methods are presented below. They differ from each other in the update of $\vx_{k,i}$.

\begin{enumerate}[A)]
	\item \textbf{Incremental subgradient method}
	\e \label{eq:IGD}
	\vx_{k,i} = \calP_{\calC}(\vx_{k,i -1} - \mu_{k} \widetilde \nabla f_i(\vx_{k,i -1} ))  \quad \text{with} \quad  \widetilde \nabla f_i(\vx_{k,i -1} ) \in \partial f_i(\vx_{k,i -1} ),
	\ee
	for $i = 1,\ldots, m$.
	
	\item \textbf{Incremental proximal point method}
	\e\label{eq:incremental proximal point}
	\vx_{k,i} =\argmin_{\vx\in \calC} \ \left\{ f_i(\vx) +  \frac{1}{2\mu_k} \|\vx-\vx_{k,i-1}\|_2^2 \right\},
	\ee
	for $i = 1,\ldots, m$.
	
	\item \textbf{Incremental  prox-linear method}
	
    When the objective function $f$ in problem \eqref{eq:problem} has the composite form  \eqref{eq:composit opt}, one can `inner-linearize' $c_i$ at $\vx_{k,i-1}$ as
	\e \label{eq:local linearization}
	f_i(\vx; \vx_{k,i-1}) = h_i\left(c_i(\vx_{k,i-1}) + \nabla c_i(\vx_{k,i-1})^\top (\vx - \vx_{k,i-1}) \right).
	\ee
	 Then, the incremental prox-linear method updates $\vx_{k,i}$ as
	\e\label{eq:incremental prox-linear}
	\vx_{k,i} = \argmin_{\vx\in \calC} \ \left\{ f_i(\vx; \vx_{k,i-1})  + \frac{1}{2\mu_k} \|\vx-\vx_{k,i-1}\|_2^2 \right\},
	\ee
	for $i = 1,\ldots, m$.
	It is worth noting that subproblem \eqref{eq:incremental prox-linear} admits a closed-form solution when $h_i$ is the absolute value function  and $\calC\equiv \R^n$, which is the case for all the concrete examples listed in \Cref{sec:mov}. Indeed,  in this case,  subproblem \eqref{eq:incremental prox-linear} takes the form
	\[ %\e\label{eq:ipl subproblem}
	\vx_{k,i} = \argmin_{\vx\in\R^n} \left\{  \left| \langle \va, \vx\rangle + b \right| + \frac{1}{2} \|\vx  - \vx_{k,i-1} \|_2^2\right\}
	\] %\ee
	for some $\va\in \R^n, b\in \R$, whose solution is given by 
	\[ %\e\label{eq:ipl closed form}
	\vx_{k,i} =  \vx_{k,i-1} - \calP_{[-1,1]} (\sigma) \va  \quad \text{with} \quad \sigma  = \frac{\langle \va, \vx_{k,i-1}\rangle + b}{\|\va\|_2^2}.
	\] %\ee
	Here, $\calP_{[-1,1]}$ represents the projector onto the line interval $[-1,1]$; see, e.g., \cite{duchi2018stochastic}.
\end{enumerate}

%To see this, we firstly suppose $\sigma$  is weakly convex with parameter $\tau$,
%then $\partial \left( \sigma + \tfrac{\tau}{2}\|\cdot\|^2 \right)$ is the usual convex subdifferential of the convex function $\vx \mapsto \sigma(\vx) + \tfrac{\tau}{2} \| \vx \|^2$ and hence the subdifferential is monotone; see, e.g., [Proposition 8.12 and Theorem 12.17] in \cite{rockafellar2009variational}. This, together with the fact that $\partial \left( h(\vx) + \tfrac{\tau}{2}\|\vx\|^2 \right) = \partial h(\vx) + \tau \vx$; see, e.g.,~[Exercise 8.8] \cite{rockafellar2009variational}, yields \eqref{eq:weak convexity}. Conversely, if~\eqref{eq:weak convexity} holds, then it is easy to see that
%\[
%\begin{aligned}
%\sigma(\vw) + \frac{\tau}{2} \|\vw\|^2 \ge &\sigma(\vx) + \frac{\tau}{2} \|\vx\|^2 +\langle \vd + \tau\vx, \vw-\vx \rangle \\
% &\quad \forall \vw,\vx\in\R^n, \ \forall \ \vd\in\partial \sigma(\vx),
%\end{aligned}
%\]
%which implies that $\sigma$ is weakly convex with parameter $\tau$.

\section{Global Convergence\label{sec:sublinear for weak convexity}}
In this section, we study the iteration complexity of  the incremental methods for solving problem \eqref{eq:problem}. 

\subsection{Assumptions and surrogate stationarity measure}
The incremental methods presented in the last section exploit different types of problem structure. Thus, we need different assumptions when studying the convergence behaviors of these methods. To start, we  state the assumptions needed in our analysis of the incremental subgradient and proximal point methods.
\begin{assum}[incremental subgradient and proximal point methods]\label{assum:subgradient and proximal point}
	\begin{itemize}
		\item (bounded subgradients) For  $i \in \{1,\ldots,m\}$,  there exist an open convex set $\calQ$ that contains $\calC$ and a constant $L_1 >0$ such that $\|\widetilde \nabla f_i(\vx)\|_2  \leq L_1$ for all $\vx\in \calQ$ and $\widetilde \nabla f_i(\vx) \in \partial f_i(\vx)$.
		
		\item (weak convexity) The component functions in \eqref{eq:problem} are weakly convex with parameter $\tau_1\ge0$. 
	\end{itemize}
\end{assum}

Note that the bounded subgradients assumption is standard in the analysis of incremental, stochastic, and subgradient-based algorithms; see, e.g., \cite{nedic2001convergence,bertsekas2011incremental,nedic2001incremental,davis2018stochastic,nemirovski2009robust,nesterov2013introductory}. In many important applications, the set $\calC$ is compact. In this case, the bounded subgradients assumption is automatically satisfied due to \eqref{eq:subdifferential} and the fact that $\partial \phi(\vx)$ is compact if $\calC$ is compact (see, e.g., \cite[Proposition B.24]{Bertsekas1999}). Even if $\calC$ is not compact, it may be possible to intersect it with a ball that is large enough to contain an optimal solution to problem~\eqref{eq:problem}, so that the bounded subgradients assumption is satisfied.

Though the component functions in  \eqref{eq:composit opt} are automatically weakly convex, the incremental prox-linear method \eqref{eq:incremental prox-linear}  explicitly exploits their composite structure. Therefore, in addition to weak convexity,
we need to make a slightly stronger assumption, namely quadratic approximation, when analyzing this algorithm. 
\begin{assum}[incremental prox-linear method]\label{assum:prox-linear}
	\begin{itemize}
		\item (bounded subgradients) For  $i \in \{1,\ldots,m\}$, there exist an open convex set $\calQ$ that contains $\calC$ and a constant $ L_2 >0$ such that $\|\widetilde \nabla f_i(\vx; \overline \vx)\|_2 \leq  L_2 $ for all $\vx,\overline \vx\in \calQ$ and $\widetilde \nabla f_i(\vx; \overline \vx) \in \partial f_i(\vx; \overline \vx)$.
		
		\item (quadratic approximation)  There exists a constant $\tau_2>0$ such that each component function $f_i$ satisfies
		\[
		f_i(\vx; \overline \vx) - f_i(\vx)   \leq \frac{\tau_2}{2} \| \vx - \overline \vx\|^2_2, \quad \forall \ \vx, \overline \vx \in \R^n.
		\]
	\end{itemize}	
Here,  $f_i(\vx; \overline \vx)$ is defined in \eqref{eq:local linearization}.
\end{assum}

The quadratic approximation assumption has been used in \cite{duchi2017solving} and \cite{davis2018stochastic,duchi2018stochastic} to analyze the full and stochastic prox-linear methods, respectively. It is straightforward to verify that this assumption immediately implies the $\tau_2$-weak convexity of $f_i$ by applying the convex subgradient inequality  to  $h_i$. To unify the notation in \Cref{assum:subgradient and proximal point} and \Cref{assum:prox-linear}, we define
\[ %\e\label{eq:L and tau}
 L = \max\{ L_1, L_2 \} \quad \text{and} \quad \tau = \max\{ \tau_1, \tau_2 \}.
\] %\ee

One of the main challenges in analyzing the convergence rate of an algorithm for nonsmooth nonconvex optimization is to find an appropriate stationarity measure to track the progress of the algorithm.
The recent  papers \cite{davis2018stochastic,drusvyatskiy2018efficiency} show that the Moreau envelope of a weakly convex function can be regarded  as an approximate Lyapunov function  and defines a surrogate stationarity measure for the weakly convex function at hand.  Once this Lyapunov function  is identified, the convergence analysis of subgradient-type methods for weakly convex minimization (see, e.g., \cite{davis2018stochastic})  essentially follows that for convex optimization. In this paper, we  adopt the said surrogate stationarity measure for analyzing the global iteration complexity of our incremental methods.  For completeness, we  briefly introduce the related notions of Moreau envelope and proximal mapping; see \cite[Definition 1.22]{rockafellar2009variational} for details.

For any $\lambda>0$, the Moreau envelope of $f$ is defined as
	\e\label{eq:moreau envelop}
	f_{\lambda} (\vx) :=  \min_{\vy\in \calC} \  \left\{ f(\vy) + \frac{1}{2\lambda}\|\vy-\vx\|_2^2 \right\},  \quad \vx\in \calC.
	\ee
	The corresponding proximal mapping is  defined as
	\e\label{eq:prox map}
	P_{\lambda, f} (\vx) :=  \argmin_{\vy\in \calC} \ \left\{ f(\vy) + \frac{1}{2\lambda}\|\vy-\vx\|^2_2 \right\} ,  \quad \vx\in \calC.
	\ee
	
By the subdifferential calculus  of sum of (regular) functions \cite[Corollary 10.9]{rockafellar2009variational}, the first-order optimality condition of $	P_{\lambda, f} (\vx)$ in \eqref{eq:prox map} implies that
\[
   \mathbf{0} \in \partial f(P_{\lambda, f} (\vx) ) + N_{\calC}(P_{\lambda, f} (\vx) ) + \frac{1}{\lambda} ( P_{\lambda, f} (\vx) - \vx).
\]
It follows that 
\e\label{eq:surrogate stationarity}
\boxed{    
	\begin{split}
		\dist\big(\mathbf{0}, \partial f(P_{\lambda, f} (\vx) ) + N_{\calC}(P_{\lambda, f} (\vx) ) \big) \leq \frac{1}{\lambda} \|\vx - P_{\lambda, f} (\vx) \|_2 =: \Theta(\vx).
	\end{split}
}
\ee
One can see from \eqref{eq:first order optimality} and \eqref{eq:surrogate stationarity} that if $\Theta(\vx) = 0$, then  $\vx\in \calC$ is a stationary point of problem \eqref{eq:problem}.  Thus, we use $\vx\mapsto \Theta(\vx)  $ as the surrogate stationarity measure of problem \eqref{eq:problem} and call $\vx\in \calC$ an $\varepsilon$-nearly stationary point if  $\Theta(\vx)\leq \varepsilon$.

\subsection{Several useful lemmas} 
Before stating our main convergence results, let us present several useful lemmas that will be used in our later development. 

First, we show that \eqref{eq:incremental proximal point} and \eqref{eq:incremental prox-linear} can be interpreted as certain subgradient-type update. 

\begin{lem}[subgradient-type update]\label{lem:rewrite ipp update}
	The following hold for all $k\ge0$ and $1\le i\le m$:
	\begin{enumerate}[(1)]
		\item If $\vx_{k,i}$ is generated by the incremental proximal point method \eqref{eq:incremental proximal point}, then there exists a  $\widetilde{\nabla} f_i(\vx_{k,i}) \in \partial f_i(\vx_{k,i})$ such that
		\e\label{eq:rewrite ipp update}
		\vx_{k,i}  = \calP_{\calC} (\vx_{k,i-1} - \mu_k \widetilde{\nabla} f_i(\vx_{k,i}) ).
		\ee
		\item If $\vx_{k,i}$ is generated by the incremental prox-linear method \eqref{eq:incremental prox-linear}, then there exists a  $\widetilde \nabla f_i(\vx_{k,i}; \vx_{k,i-1}) \in \partial f_i(\vx_{k,i}; \vx_{k,i-1}) $ such that
		\e\label{eq:rewrite ipl update}
		\vx_{k,i}  = \calP_{\calC} ( \vx_{k,i-1} - \mu_k  \widetilde \nabla f_i(\vx_{k,i}; \vx_{k,i-1}) ).
		\ee
%		where $f_i(\vx)$ has the composite form in \eqref{eq:composit opt} and $f_i(\vx; \vx_{k,i-1})$ is a local convexification of $f_i(\vx)$ at $\vx_{k,i-1}$; see \eqref{eq:local linearization}.
	\end{enumerate}
\end{lem}

\begin{proof}%[Proof of \Cref{lem:rewrite ipp update}]
	We provide the proof for \eqref{eq:rewrite ipp update}. The proof of \eqref{eq:rewrite ipl update} follows a similar argument. According to the first-order optimality of $\vx_{k,i} $ in \eqref{eq:incremental proximal point} and \cite[Corollary 10.9]{rockafellar2009variational}, we have 
	\[ %\e\label{eq:first order proximal point}
	    \frac{1}{\mu_k} (\vx_{k,i -1} - \vx_{k,i}) \in \partial f_i (\vx_{k,i}) + N_{\calC}(\vx_{k,i}),
	\] %\ee
	which implies that
   \[
	(\vx_{k,i -1} - \mu_k \widetilde \nabla f_i(\vx_{k,i}) )  -   \vx_{k,i} \in N_{\calC}(\vx_{k,i})
	\]
	for some $\widetilde \nabla f_i(\vx_{k,i})  \in \partial f_i(\vx_{k,i})$. The above inclusion is  equivalent to \eqref{eq:rewrite ipp update} by  convexity of $\calC$.
\end{proof}

The relations  \eqref{eq:rewrite ipp update} and \eqref{eq:rewrite ipl update}
are crucial to the analysis of proximal-type algorithms  \cite{bertsekas2011incremental,bertsekas2015incremental}. It is interesting to see that the updates \eqref{eq:rewrite ipp update} and \eqref{eq:rewrite ipl update} are very similar to that of the incremental subgradient method \eqref{eq:IGD}. The only difference is that the subgradients in \eqref{eq:rewrite ipp update} and \eqref{eq:rewrite ipl update} are evaluated at $\vx_{k,i}$, while that in \eqref{eq:IGD} is evaluated at $\vx_{k,i-1}$. This observation suggests that we can follow a unified proof strategy for all  three incremental methods.

Next, we develop a preliminary recursion for the term $\|\vx_{k,i} -\vy\|_2^2$, where $\vx_{k,i}$ is the iterate generated by one of the  incremental methods and $\vy\in \calC$ is arbitrary. 

\begin{lem}[preliminary recursion]\label{lem:basic recursion}
	For any $\vy\in\calC$, $k\geq 0$, and $1\leq i\leq m$, we have
		\e\label{eq:preliminary recursion}
	\begin{aligned}
		\|\vx_{k,i} -\vy \|^2_2 & \leq \|\vx_{k,i-1} -\vy \|^2_2- 2\mu_k(f_{i}(\vz_i)-f_{i}(\vy))\\
		&\quad  + \tau \mu_k \|\vz_i - \vy \|^2_2+ \gamma \|\vx_{k,i} - \vx_{k,i-1}\|_2^\theta, 
	\end{aligned}
	\ee
	
	where  	
		\begin{enumerate}[(a)]
		\item 
		     \[
		       \vz_i = \vx_{k,i-1}, \quad \gamma = 1, \quad \theta = 2
		     \]
		if $\vx_{k,i}$ is generated by the incremental subgradient method \eqref{eq:IGD};
	             
		\item
		      \[
		     \vz_i = \vx_{k,i}, \quad \gamma = 0
		     \]
			if $\vx_{k,i}$ is generated by the incremental proximal point method \eqref{eq:incremental proximal point};

		\item  
		\[
		\vz_i = \vx_{k,i-1}, \quad \gamma = 2\mu_kL, \quad \theta = 1
		\]
		 	if $\vx_{k,i}$ is generated by  the incremental prox-linear method \eqref{eq:incremental prox-linear}.

	\end{enumerate}

%	\begin{enumerate}[(a)]
%		\item If $\vx_{k,i}$ is generated by incremental subgradient method \eqref{eq:IGD}, we have 
%		\e\label{eq:basic recursion IGD}
%		\begin{aligned}
%			\|\vx_{k,i} -\vy \|^2_2 & \leq \|\vx_{k,i-1} -\vy \|^2_2- 2\mu_k(f_{i}(\vx_{k,i-1})-f_{i}(\vy))\\
%			&\quad  + \tau \mu_k \|\vx_{k,i-1} - \vy \|^2_2+ \mu_{k}^2 L^2.
%		\end{aligned}
%		\ee
%		\item If $\vx_{k,i}$ is generated by incremental proximal point method \eqref{eq:incremental proximal point}, we have 
%		\e\label{eq:basic recursion IPP}
%		\begin{aligned}
%			\|\vx_{k,i} - \vy\|^2_2 \leq  \|\vx_{k,i-1} - \vy\|_2^2 - 2 \mu_k ( f_i(\vx_{k,i}) - f_i(\vy))  + \tau \mu_k \|\vx_{k,i} - \vy\|_2^2.
%		\end{aligned}
%		\ee
%		
%		\item  If $\vx_{k,i}$ is generated by incremental prox-linear method \eqref{eq:incremental prox-linear}, we have 
%		\e\label{eq:basic recursion IPL}
%		\begin{aligned}
%			\|\vx_{k,i} - \vy\|^2_2 & \leq  \|\vx_{k,i-1} - \vy\|_2^2 - 2\mu_k  ( f_i(\vx_{k,i-1}) - f_i(\vy) ) \\
%			&\quad + \tau \mu_k   \|\vx_{k,i-1} - \vy\|_2^2  + 2\mu_k L  \|\vx_{k,i} - \vx_{k,i-1} \|_2.
%		\end{aligned}
%		\ee
%	\end{enumerate}
\end{lem}

\begin{proof}%[Proof of \Cref{lem:basic recursion}]
	We  decompose the error $\|\vx_{k,i-1} -\vy \|^2_2 = \|\vx_{k,i-1} - \vx_{k,i} + \vx_{k,i} - \vy \|^2_2$. Expanding the right-hand side, we have 
	\e\label{eq:preliminary recursion 1}
		\|\vx_{k,i} -\vy \|^2_2 =  \|\vx_{k,i-1} -\vy \|^2_2 - 2 \left\langle  \vx_{k,i-1} - \vx_{k,i} , \vx_{k,i} - \vy  \right\rangle  -\|\vx_{k,i} - \vx_{k,i-1}\|_2^2.
	\ee
	 
	We first show (a). 	By the first-order optimality of $\vx_{k,i} $ in \eqref{eq:IGD}, there exists an $\vs\in   N_{\calC} (\vx_{k,i}) $  such that $\vx_{k,i-1} - \vx_{k,i} = \mu_k (\widetilde\nabla f_i(\vx_{k,i-1}) + \vs) $.  Plugging this relation into \eqref{eq:preliminary recursion 1} yields 
	\e\label{eq:preliminary recursion ISG 1}
	\begin{aligned}
		\|\vx_{k,i} -\vy \|^2_2 &=   \|\vx_{k,i-1} -\vy \|^2_2 - 2\mu_k \left\langle \widetilde\nabla f_i(\vx_{k,i-1}) +\vs , \vx_{k,i} - \vy  \right\rangle   -\|\vx_{k,i} - \vx_{k,i-1}\|_2^2\\
		& \leq  \|\vx_{k,i-1} -\vy \|^2_2 - 2\mu_k \left\langle \widetilde\nabla f_i(\vx_{k,i-1}) , \vx_{k,i-1} - \vy  \right\rangle   \\
		& \quad - 2\mu_k \left\langle \widetilde\nabla f_i(\vx_{k,i-1}) + \vs , \vx_{k,i} - \vx_{k,i-1} \right\rangle   -\|\vx_{k,i} - \vx_{k,i-1}\|_2^2,\\
		&=  \|\vx_{k,i-1} -\vy \|^2_2 - 2\mu_k \left\langle \widetilde\nabla f_i(\vx_{k,i-1}) , \vx_{k,i-1} - \vy  \right\rangle  + \|\vx_{k,i} - \vx_{k,i-1}\|_2^2,
	\end{aligned}
	\ee
	where we have used the fact that $- \left\langle \vs , \vx_{k,i} - \vy  \right\rangle \leq 0$ in the inequality. 
	By using the weak convexity assumption in \Cref{assum:subgradient and proximal point} and applying the weakly convex inequality \eqref{eq:weak convexity} to \eqref{eq:preliminary recursion ISG 1}, we have 
	\[ %\e\label{eq:preliminary recursion ISG 2}
	\begin{aligned}
			\|\vx_{k,i} -\vy \|^2_2 &\leq \|\vx_{k,i-1} -\vy \|^2_2- 2\mu_k(f_{i}(\vx_{k,i-1})-f_{i}(\vy)) \\
		&\quad + \tau \mu_k \|\vx_{k,i-1} - \vy \|^2_2 + \|\vx_{k,i} - \vx_{k,i-1}\|_2^2.
	\end{aligned}
	\] %\ee
	
	We now show (b). By the first-order optimality  of $\vx_{k,i} $ in \eqref{eq:incremental proximal point} and \cite[Corollary 10.9]{rockafellar2009variational}, there exist $\widetilde\nabla f_i(\vx_{k,i}) \in \partial f_i(\vx_{k,i})$ and $\vs\in   N_{\calC} (\vx_{k,i}) $  such that $\vx_{k,i-1} - \vx_{k,i} = \mu_k (\widetilde\nabla f_i(\vx_{k,i})+ \vs) $. Invoking this relation in \eqref{eq:preliminary recursion 1} gives
	\[ %\e\label{eq:basic recursion lemma 2}
	\begin{aligned}
		\|\vx_{k,i}& - \vy\|^2_2 =  \|\vx_{k,i-1} - \vy\|_2^2 - 2\mu_k \left\langle \widetilde{\nabla} f_i(\vx_{k,i}) + \vs, \vx_{k,i} - \vy \right\rangle - \|\vx_{k,i} - \vx_{k,i-1}\|_2^2\\
		& \leq \|\vx_{k,i-1} - \vy\|_2^2 - 2\mu_k \left\langle \widetilde{\nabla} f_i(\vx_{k,i}), \vx_{k,i} - \vy \right\rangle   - \|\vx_{k,i} - \vx_{k,i-1}\|_2^2 \\
		&\leq  \|\vx_{k,i-1} - \vy\|_2^2 - 2 \mu_k ( f_i(\vx_{k,i}) - f_i(\vy)) + \tau\mu_k \|\vx_{k,i} - \vy\|_2^2,
	\end{aligned}
	\] %\ee
	where  the last inequality is due to the weak convexity assumption in \Cref{assum:subgradient and proximal point} and \eqref{eq:weak convexity}.
	
	Finally, we show (c).  The first-order optimality  of $\vx_{k,i} $ in \eqref{eq:incremental prox-linear} and  \cite[Corollary 10.9]{rockafellar2009variational} ensure the existence of $\widetilde{\nabla} f_i(\vx_{k,i};\vx_{k,i -1})  \in \partial f_i(\vx_{k,i};\vx_{k,i -1}) $ and $\vs\in   N_{\calC} (\vx_{k,i}) $  such that $\vx_{k,i-1} - \vx_{k,i} = \mu_k (\widetilde\nabla f_i(\vx_{k,i};\vx_{k,i -1}) + \vs) $. This, together with \eqref{eq:preliminary recursion 1}, gives
	\e\label{eq:basic recursion lemma 3}
	\begin{aligned}
			&\|\vx_{k,i} - \vy\|^2_2 \\
			&=  \|\vx_{k,i-1} - \vy\|_2^2 - 2\mu_k \left\langle \widetilde{\nabla} f_i(\vx_{k,i};\vx_{k,i -1}) + \vs, \vx_{k,i} - \vy \right\rangle  - \|\vx_{k,i} - \vx_{k,i-1}\|_2^2\\
		&\leq  \|\vx_{k,i-1} - \vy\|_2^2 - 2\mu_k \left\langle \widetilde{\nabla} f_i(\vx_{k,i};\vx_{k,i-1}), \vx_{k,i} - \vy \right\rangle   - \|\vx_{k,i} - \vx_{k,i-1}\|_2^2.
	\end{aligned}
	\ee
	 Recalling that $\vx \mapsto f_i(\vx;\vx_{k,i -1})$ is convex, we have
	\e\label{eq:basic recursion lemma 4}
	\begin{aligned}
		&- \left\langle \widetilde \nabla f_i(\vx_{k,i}; \vx_{k,i-1}),  \ \vx_{k,i} - \vy \right\rangle \leq f_i(\vy; \vx_{k,i-1}) - f_i(\vx_{k,i}; \vx_{k,i-1})\\
		&\leq f_i(\vy) + \frac{\tau}{2} \|\vx_{k,i-1} - \vy\|_2^2 - f_i(\vx_{k,i-1})  +f_i(\vx_{k,i-1})  - f_i(\vx_{k,i}; \vx_{k,i-1})\\
		&\leq - ( f_i(\vx_{k,i-1}) - f_i(\vy) ) + \frac{\tau}{2} \|\vx_{k,i-1} - \vy\|_2^2 + {L} \|\vx_{k,i} - \vx_{k,i-1} \|_2,
	\end{aligned}
	\ee
	where the second line utilizes the quadratic approximation assumption in \Cref{assum:prox-linear} and  the last inequality follows from the bounded subgradients assumption in \Cref{assum:prox-linear} and \cite[Theorem 9.13]{rockafellar2009variational}.  Substituting \eqref{eq:basic recursion lemma 4} into \eqref{eq:basic recursion lemma 3} yields
	\[ %\e\label{eq:basic recursion lemma 5}
	\begin{aligned}
		\|\vx_{k,i} - \vy\|^2_2 & \leq  \|\vx_{k,i-1} - \vy\|_2^2 - 2\mu_k  ( f_i(\vx_{k,i-1}) - f_i(\vy) ) + \tau \mu_k   \|\vx_{k,i-1} - \vy\|_2^2 \\
		& \quad + 2\mu_k L  \|\vx_{k,i} - \vx_{k,i-1} \|_2.
	\end{aligned}
	\] %\ee
	This completes the proof of \Cref{lem:basic recursion}.
\end{proof}

%In the following lemma, we present a upper bound on the length  between inner iterates and the starting point in $k+1$-th iteration, i.e., $\|\vx_{k,i} - \vx_k\|_2$, for all the three incremental methods. 
\begin{lem}[inner step length]\label{lem:inner iteration length}
	 For any $k\geq 0$, let $\{\vx_{k,i}\}_{i=1}^m$ be   the sequence  generated by any of  the three incremental methods. Then, we have 
	\e\label{eq:inner iteration length}
	\| \vx_{k,i} - \vx_{k,j}\|_2  \leq |i-j| \mu_k L, \quad \forall \ i,j\in \{1,\ldots,m\}.
	\ee
\end{lem}
\begin{proof}%[Proof of \Cref{lem:inner iteration length}]
	Without loss of generality, we assume that $i\geq j$. According to \Cref{lem:rewrite ipp update}, the updates of the three incremental methods can be written in a unified manner as $\vx_{k,i} = \calP_{\calC}(\vx_{k,i-1} - \mu_{k} \mS_{k,i})$, where $\mS_{k,i} = \widetilde\nabla f_i(\vx_{k,i-1})\in \partial f_i(\vx_{k,i-1})$ for the incremental subgradient method,   $\mS_{k,i} = \widetilde\nabla f_i(\vx_{k,i})\in \partial f_i(\vx_{k,i})$ for the incremental proximal point method, and  $\mS_{k,i} = \widetilde\nabla f_i(\vx_{k,i};\vx_{k,i-1})\in \partial f_i(\vx_{k,i};\vx_{k,i-1})$ for the incremental prox-linear method.  Now, we prove \eqref{eq:inner iteration length} by induction.  It is trivial to verify that \eqref{eq:inner iteration length}  is true when $i=j$.  Assuming that \eqref{eq:inner iteration length}  holds for $i = l$,  we have 
	\[ 
	\begin{aligned}
	       \| 	\vx_{k,l+1} - 	\vx_{k,j}\|_2 &= \| \calP_{\calC}(	\vx_{k,l} - \mu_k \mS_{k,l+1})     - 	\vx_{k,j}\|_2 \\
	       &\leq \|\vx_{k,l} - \mu_k \mS_{k,l+1}  - 	\vx_{k,j}\|_2 \\
	       &\leq \|\vx_{k,l}   - 	\vx_{k,j}\|_2  + \mu_k\| \mS_{k,l+1}\|_2 \\
	       &\leq (l+1 -j ) \mu_k L,
	\end{aligned}
	\]
	where the first inequality is due to the fact that the projector $\calP_{\calC}$ is nonexpansive; the last inequality is due to $\|\mS_{k,l}\|_2\leq L$, which follows from the bounded subgradients assumptions in \Cref{assum:subgradient and proximal point} and \Cref{assum:prox-linear}. 
\end{proof}

\subsection{Global sublinear convergence result}
The following proposition provides  an important recursion shared by  all three incremental methods.  
\begin{prop}[key recursion for global convergence]\label{prop:global convergence}
	Suppose that \Cref{assum:subgradient and proximal point} is valid when considering the incremental subgradient and proximal point methods, while \Cref{assum:prox-linear} is valid when considering the incremental prox-linear method. Let  $\{\vx_{k}\}_{k\geq 0}$ be  the sequence generated by any of  the three incremental methods for solving problem \eqref{eq:problem} with arbitrary initialization.  Then, for any $\lambda < \frac{1}{2\tau}$ in \eqref{eq:moreau envelop},  we have
	\e\label{eq:global convergence recursion}
	\begin{aligned}
		f_{\lambda}(\vx_{k+1}) & \leq f_{\lambda}(\vx_{k}) - \left(\frac{1}{2\lambda} -  \tau\right) \frac{1}{\lambda} m \mu_k  \| \vx_k - P_{\lambda,f}(\vx_{k}) \|^2_2 \\
		&\quad + \left( 1+ \frac{1/\lambda}{1/\lambda - \tau}\right) \frac{1}{\lambda}  m^2 \mu_k^2 L^2 + \left( 1+ \frac{1/\lambda}{1/\lambda - \tau}\right)^2  \frac{1}{\lambda}\tau  m^3  \mu_k^3 L^2.
	\end{aligned}
	\ee 
\end{prop}

\begin{proof}%[Proof of \Cref{prop:global convergence}]
	
	From the optimality of $P_{\lambda,f}(\vx_{k,i})$ in \eqref{eq:prox map} and the fact that $P_{\lambda, f} ( \bm{x}_{k,i-1} ) \in \mathcal{C}$, we have 
	\[ %\e\label{eq:optimality}
	f_{\lambda}(\vx_{k,i}) \leq f(P_{\lambda,f}(\vx_{k,i-1})) + \frac{1}{2 \lambda } \| \vx_{k,i} - P_{\lambda,f}(\vx_{k,i-1}) \|_2^2.
	\] %\ee
	 This, together with  \Cref{lem:basic recursion} (by letting $\bm{y} = P_{\lambda, f} ( \bm{x}_{k,i-1} )$) and the definition of the Moreau envelope $f_{\lambda} ( \bm{x}_{k, i-1} ) = f ( P_{\lambda, f} ( \bm{x}_{k, i-1} ) ) + \tfrac{1}{2\lambda} \| \bm{x}_{k, i-1} - P_{\lambda, f} ( \bm{x}_{k, i-1} ) \|_2^2$, gives  
	\e \label{eq:recursion 1}
	\begin{aligned}
	f_{\lambda}(\vx_{k,i}) & \leq f_{\lambda}(\vx_{k,i-1})  -  \frac{ \mu_k}{\lambda} \big( f_i(\vz_i) -  f_i(P_{\lambda,f}(\vx_{k,i-1})) \big)  \\
	&\quad  + \frac{ \tau \mu_k}{2\lambda}    \|\vz_i - P_{\lambda,f}(\vx_{k,i-1}) \|_2^2 + \frac{\gamma}{2\lambda} \|\vx_{k,i} - \vx_{k,i-1}\|_2^\theta.
	\end{aligned}
	\ee
	 Summing the inequality in \eqref{eq:recursion 1} over $i=1,\ldots,m$ yields
	\e\label{eq:recursion 2}
	\begin{aligned}
		f_{\lambda}(\vx_{k+1}) & \leq  f_{\lambda}(\vx_{k})   -  \frac{ \mu_k}{\lambda} \underbrace{\sum_{i = 1}^m  \big( f_i(\vz_i) -  f_i(P_{\lambda,f}(\vx_{k,i-1})) \big) }_{\Delta_1} \\
		&\quad  +  \frac{ \tau \mu_k}{2\lambda}    \underbrace{ \sum_{i = 1}^m \|\vz_i - P_{\lambda,f}(\vx_{k,i-1})\|_2^2}_{\Delta_2} + \underbrace{\frac{\gamma}{2\lambda} \sum_{i = 1}^m \|\vx_{k,i} - \vx_{k,i-1}\|_2^\theta }_{\Delta_3}.
	\end{aligned}
	\ee
	We now bound $\Delta_1$, $\Delta_2$, and $\Delta_3$ in \eqref{eq:recursion 2} for different incremental methods.

	\vspace{0.2cm}
	\emph{Part I:  Incremental subgradient method}. According to \Cref{lem:basic recursion}, we have $\vz_i = \vx_{k,i-1}$, $\gamma = 1$, $\theta = 2$ in \eqref{eq:recursion 2} for the incremental subgradient method. For $\Delta_1$, we have
	\e\label{eq:Delta 1}
	\begin{aligned}
		\Delta_1 &=  \sum_{i = 1}^m [ f_i(\vx_{k,i-1}) - f_i(\vx_k)
		+ f_i(P_{\lambda,f}(\vx_{k})) \\
		&\quad -  f_i(P_{\lambda,f}(\vx_{k,i-1})) + f_i(\vx_k)  - f_i(P_{\lambda,f}(\vx_{k}))   ] \\
		& \geq  m [f(\vx_k) - f(P_{\lambda,f}(\vx_{k}))] \\
		&\quad - L \sum_{i = 1}^m \big(\| \vx_{k,i-1} - \vx_k\|_2 +   \| P_{\lambda,f}(\vx_{k,i-1}) - P_{\lambda,f}(\vx_{k})\|_2 \big) \\
		&\geq m [f(\vx_k) - f(P_{\lambda,f}(\vx_{k}))]   -L\left( 1+ \frac{1/\lambda}{1/\lambda - \tau}\right) \sum_{i = 1}^m \| \vx_{k,i-1} - \vx_k\|_2 \\
		&\geq m [f(\vx_k) - f(P_{\lambda,f}(\vx_{k}))] - \Big( 1+ \frac{1/\lambda}{1/\lambda - \tau}\Big)  \frac{m^2-m}{2}   \mu_k L^2, 
	\end{aligned}
	\ee
	where the first inequality utilizes the bounded subgradients assumption in \Cref{assum:subgradient and proximal point} and \cite[Theorem 9.13]{rockafellar2009variational}; the second inequality follows from \cite[Proposition 12.19]{rockafellar2009variational},  which states that for any $\tau$-weakly convex function $f$, $P_{\lambda,f}$ is Lipschitz continuous with constant ${(1/\lambda)}/((1/\lambda) - \tau)$ if $\lambda < \frac{1}{\tau}$;  the last inequality is because of  \Cref{lem:inner iteration length}. 
	
	Similarly,
	\e \label{eq:Delta 2 v1}
	\begin{aligned}
	\Delta_2 &= \sum_{i = 1}^m \|\vx_{k,i-1} - \vx_k + P_{\lambda,f}(\vx_{k}) - P_{\lambda,f}(\vx_{k,i-1}) +  \vx_k - P_{\lambda,f}(\vx_{k}) \|_2^2 \\
	&\leq  2\sum_{i = 1}^m \left( \left( 1+ \frac{1/\lambda}{1/\lambda - \tau}\right)^2\|\vx_{k,i-1} - \vx_k\|_2^2 + \| \vx_k - P_{\lambda,f}(\vx_{k}) \|_2^2 \right),
	\end{aligned}
	\ee
	where the last line utilizes the same Lipschitz continuous property of $P_{\lambda,f}$  as in \eqref{eq:Delta 1}.  By upper bounding $\|\vx_{k,i-1} - \vx_k\|_2$ using \Cref{lem:inner iteration length}, one can see that
	\e\label{eq:Delta 2 v2}
	\begin{aligned}
		\Delta_2 &\leq \left( 1+ \frac{1/\lambda}{1/\lambda - \tau}\right)^2  \frac{(m-1)m(2m-1)}{3} \mu_k^2  L^2 + 2 m\| \vx_k - P_{\lambda,f}(\vx_{k}) \|_2^2.
	\end{aligned}
	\ee
	
	For $\Delta_3$, \Cref{lem:inner iteration length} gives 
	\e\label{eq:Delta 3 v1}
	  \Delta_3 \leq \frac{1}{2\lambda} m \mu_{k}^2 L^2.
 	\ee
	
	Substituting \eqref{eq:Delta 1},  \eqref{eq:Delta 2 v2}, and \eqref{eq:Delta 3 v1} into \eqref{eq:recursion 2} yields 
	\e\label{eq:recursion 4}
	\begin{aligned}
		&f_{\lambda}(\vx_{k+1})  \leq f_{\lambda}(\vx_{k}) -  \frac{ m \mu_k}{\lambda}  \left( f(\vx_k) - f(P_{\lambda,f}(\vx_{k}))  - \tau\| \vx_k - P_{\lambda,f}(\vx_{k}) \|_2^2\right) \\
		&\quad + \left( 1+ \frac{1/\lambda}{1/\lambda - \tau}\right) \frac{1}{\lambda}  \frac{m^2+m}{2} \mu_k^2 L^2 + \left( 1+ \frac{1/\lambda}{1/\lambda - \tau}\right)^2 \frac{\tau}{\lambda}  \frac{(m-1)m(2m-1)}{6}  \mu_k^3 L^2,
	\end{aligned}
	\ee
	where we have enlarged the term $\frac{1}{2\lambda} m \mu_{k}^2  L^2$ in \eqref{eq:Delta 3 v1} to $\frac{1}{\lambda}\left( 1+ \frac{1/\lambda}{1/\lambda - \tau}\right) m \mu_k^2 L^2$.
	Note that
	\e \label{eq:me property}
	\begin{aligned}
	&f(\vx_k) - f(P_{\lambda,f}(\vx_{k}))  - \tau\| \vx_k - P_{\lambda,f}(\vx_{k}) \|_2^2 \\
	&= f(\vx_k) - \underbrace{\Big( f(P_{\lambda,f}(\vx_{k})) + \frac{1}{2\lambda}\| \vx_k -P_{\lambda,f}(\vx_{k})\|_2^2 \Big) }_{f_{\lambda}(\vx_{k})}  +\left(\frac{1}{2\lambda} -  \tau\right) \| \vx_k - P_{\lambda,f}(\vx_{k}) \|_2^2\\
	&\geq \left(\frac{1}{2\lambda} -  \tau\right) \| \vx_k - P_{\lambda,f}(\vx_{k}) \|_2^2,
	\end{aligned}
	\ee
	where the inequality is from the definition of the Moreau envelope. Plugging the above inequality into \eqref{eq:recursion 4}  provides 
	\e\label{eq:recursion 4-2}
	\begin{aligned}
		f_{\lambda}(\vx_{k+1})  & \leq f_{\lambda}(\vx_{k}) - \left(\frac{1}{2\lambda} -  \tau\right) \frac{ m \mu_k}{\lambda}  \| \vx_k - P_{\lambda,f}(\vx_{k}) \|_2^2 \\
		&\quad +  \left( 1+ \frac{1/\lambda}{1/\lambda - \tau}\right) \frac{1}{\lambda}    m^2 \mu_k^2 L^2 + \left( 1+ \frac{1/\lambda}{1/\lambda - \tau}\right)^2 \frac{\tau}{\lambda}   m^3  \mu_k^3 L^2.
	\end{aligned}
	\ee
	Note that we have enlarged $\frac{m^2+m}{2}$ to $m^2$ and $\frac{(m-1)m(2m-1)}{6}$ to $m^3$  in order for  all  three incremental methods to satisfy this recursion. This yields the desired result for the incremental subgradient method. 
	%\widehat \tau m \left(\sum_{k = 0}^N  \mu_k\right)  \min_{0 \leq k \leq N}\left[f(\vx_k) - f(\overline \vx_k)  - \tau\| \vx_k - \overline \vx_k \|^2\right]
	
	\vspace{0.2cm}
	\emph{Part II:  Incremental proximal point method}. In this case, we have $\vz_i = \vx_{k,i}$ and $\gamma = 0$ (i.e., $\Delta_3=0$) in \eqref{eq:recursion 2}. Using  similar arguments as \eqref{eq:Delta 1} and \eqref{eq:Delta 2 v1},  we have 
	\e\label{eq:Delta 1-2}
	\begin{aligned}
		\Delta_1 &\geq m [f(\vx_k) - f(P_{\lambda,f}(\vx_{k}))]   - L \sum_{i = 1}^m \| \vx_{k,i} - \vx_k\|_2  - L \frac{1/\lambda}{1/\lambda - \tau}  \sum_{i = 1}^m \| \vx_{k,i-1} - \vx_k\|_2 \\
		&\geq m [f(\vx_k) - f(P_{\lambda,f}(\vx_{k}))] - \left( 1+ \frac{1/\lambda}{1/\lambda - \tau}\right)  \frac{m^2+m}{2}   \mu_k L^2
	\end{aligned}
	\ee
	and 
	\e \label{eq:Delta 2-2}
	\begin{aligned}
		\Delta_2 &= \sum_{i = 1}^m \|\vx_{k,i} - \vx_k + P_{\lambda,f}(\vx_{k}) - P_{\lambda,f}(\vx_{k,i-1}) +  \vx_k - P_{\lambda,f}(\vx_{k}) \|_2^2 \\
		&\leq   2\sum_{i = 1}^m  \left(  \|\vx_{k,i} - \vx_k\|_2 +   \frac{1/\lambda}{1/\lambda - \tau}  \|\vx_{k,i-1} - \vx_k\|_2 \right)^2 + 2m \| \vx_k - P_{\lambda,f}(\vx_{k}) \|_2^2 \\
		&\leq \left( 1+ \frac{1/\lambda}{1/\lambda - \tau}\right)^2  \frac{m(m+1)(2m+1)}{3} \mu_k^2  L^2 + 2 m\| \vx_k - P_{\lambda,f}(\vx_{k}) \|_2^2. 
	\end{aligned}
	\ee
	Substituting \eqref{eq:Delta 1-2} and \eqref{eq:Delta 2-2} into \eqref{eq:recursion 2} and applying \eqref{eq:me property} lead to exactly the same recursion as \eqref{eq:recursion 4-2}, where we have enlarged $\frac{m^2+m}{2}$ to $m^2$ and $\frac{m(m+1)(2m+1)}{6}$ to $m^3$. 
%	\e\label{eq:Moreau envelope}
%	\left\{
%	\begin{array}{l@{\ \ \ }l}
%		\displaystyle f_{\lambda}(\mX) = \min_{\mY\in (n,r)} \left\{ f(\mY) + \frac{1}{2\lambda}  \|\mY - \mX\|_F^2 \right\}, & \mX\in(n,r),\\
%		\noalign{\smallskip}
%		\displaystyle  P_{\lambda f} (\mX) \in \argmin_{\mY\in (n,r)} \left\{ f(\mY) + \frac{1}{2\lambda}  \|\mY - \mX\|_F^2 \right\}, & \mX\in(n,r).
%	\end{array}
%	\right.
%	\ee

	\vspace{0.2cm}
	\emph{Part III: Incremental prox-linear method}. According to \Cref{lem:basic recursion}, we have $\vz_i = \vx_{k,i-1}$, $\gamma = 2\mu_kL$, and $\theta = 1$ in \eqref{eq:recursion 2} for the incremental prox-linear method.  In this case, the bounds for $\Delta_1$ and $\Delta_2$ are exactly the same as \eqref{eq:Delta 1} and \eqref{eq:Delta 2 v2}, respectively. For $\Delta_3$, \eqref{eq:inner iteration length} gives $\Delta_3 \leq \frac{1}{\lambda} m \mu_{k}^2 L^2$. Substituting the bounds for $\Delta_1$, $\Delta_2$, and $\Delta_3$ into \eqref{eq:recursion 2}  and applying \eqref{eq:me property}, we  obtain   exactly the same recursion as \eqref{eq:recursion 4-2}. 
\end{proof}

Equipped with the above proposition, we are ready to establish our global convergence result.
\begin{thm}[global convergence]\label{thm:global convergence rate}
	 Under the  setting of \Cref{prop:global convergence},   the following hold:
	 
	 \begin{enumerate}[(a)]
	 	\item If we choose the constant stepsize $\mu_k = \frac{1}{m\tau\sqrt{T+1}}$ for $k\geq0$ with $T$ being the total number of iterations, then 
	 	\[
	 		\min_{0 \leq k \leq T}  \Theta^2(\vx_k) 
	 		\leq \frac{ C_1 } {\left(\frac{1}{2\lambda} -  \tau\right) \sqrt{T+1} }        +      \frac{ C_2 } {\left(\frac{1}{2\lambda} -  \tau\right) (T+1) }.
	 	\]
	 	
	 	\item If we choose the diminishing stepsizes $\mu_k = \frac{1}{m\tau \sqrt{k+1}}$ for $k \geq 0$, then 
	 	\[
	 	\min_{0 \leq k \leq T}  \Theta^2(\vx_k) 
	 	\leq \frac{ C_1 } {\left(\frac{1}{2\lambda} -  \tau\right) \sqrt{T+1} }        +      \frac{ C_2 (\ln(T+1) +1)} {\left(\frac{1}{2\lambda} -  \tau\right) (T+1) }.
	 	\]
	 \end{enumerate}
	Here,  $C_1 =  \frac{\tau}{\lambda} \left(f_{\lambda}(\vx_{0}) - \min f_{\lambda} \right)+ \left( 1+ \frac{1/\lambda}{1/\lambda - \tau}\right) \frac{1}{\lambda^2\tau} L^2$ and  $C_2 = \left( 1+ \frac{1/\lambda}{1/\lambda - \tau}\right)^2 \frac{1}{\lambda^2\tau} L^2$.
\end{thm}

\begin{proof}%[Proof of \Cref{thm:global convergence rate}]
	Unrolling the recursion \eqref{eq:global convergence recursion} in \Cref{prop:global convergence} and invoking the definition $ \Theta(\vx_k)  = \frac{1}{\lambda} \| \vx_k - P_{\lambda,f}(\vx_{k}) \|_2$ from \eqref{eq:surrogate stationarity}, we obtain
	\e\label{eq:recursion 5}
	\begin{aligned}
		&\left(\frac{1}{2\lambda} -  \tau\right) \left(\sum_{k = 0}^T  \mu_k\right) m \lambda \min_{0 \leq k \leq T} \Theta^2(\vx_k) 
		 \leq f_{\lambda}(\vx_{0}) - \min f_{\lambda} \\
		&\quad +  \left( 1+ \frac{1/\lambda}{1/\lambda - \tau}\right) \frac{1}{\lambda}    m^2  \left(\sum_{k = 0}^T  \mu_k^2\right) L^2
+  \left( 1+ \frac{1/\lambda}{1/\lambda - \tau}\right)^2 \frac{\tau}{\lambda}  m^3 \left(\sum_{k = 0}^T  \mu_k^3\right) L^2.
	\end{aligned}
	\ee
	%\widehat \tau m \left(\sum_{k = 0}^N  \mu_k\right)  \min_{0 \leq k \leq N}\left[f(\vx_k) - f(\overline \vx_k)  - \tau\| \vx_k - \overline \vx_k \|^2\right]
	Dividing $\left(\frac{1}{2\lambda} -  \tau\right) \left(\sum_{k = 0}^T  \mu_k\right) m\lambda $ on both sides of \eqref{eq:recursion 5} gives
	\e\label{eq:recursion 6}
	\begin{aligned}
		 \min_{0 \leq k \leq T} \Theta^2(\vx_k)  &\leq \frac{ \frac{1}{\lambda} \left(f_{\lambda}(\vx_{0}) - \min f_{\lambda} \right) +  \left( 1+ \frac{1/\lambda}{1/\lambda - \tau}\right) \frac{1}{\lambda^2}   m^2 \left(\sum_{k = 0}^T  \mu_k^2\right) L^2    }       {\left(\frac{1}{2\lambda} -  \tau\right) \left(\sum_{k = 0}^T  \mu_k\right) m } \\
		& \quad+ \frac{ \left( 1+ \frac{1/\lambda}{1/\lambda - \tau}\right)^2 \frac{\tau}{\lambda^2}  m^3 \left(\sum_{k = 0}^T \mu_k^3\right)  L^2}    {\left(\frac{1}{2\lambda} -  \tau\right) \left(\sum_{k = 0}^T  \mu_k\right) m}.
	\end{aligned}
	\ee
	The result in (a) follows by taking the constant stepsize $\mu_k= \frac{1}{m \tau \sqrt{N+1}}$ in \eqref{eq:recursion 6}. The result in (b) follows by taking the diminishing stepsizes $\mu_k = \frac{1}{m\tau \sqrt{k+1}}$  in \eqref{eq:recursion 6} and recognizing  that $\sum_{k=0}^{T} \frac{1}{\sqrt{k+1}} >\sqrt{T+1}$ and $\sum_{k=0}^{T} \frac{1}{k+1} < \ln(T+1) +1$. 
\end{proof}

\Cref{thm:global convergence rate} implies that the iteration complexity of the incremental methods for computing an $\varepsilon$-nearly stationary point of problem \eqref{eq:problem} is $\calO(\varepsilon^{-4})$, which matches that of their stochastic counterparts \cite{davis2018stochastic}.

\section{ Linear convergence for  sharp weakly convex optimization\label{sec:local linear convergence}}  The global sublinear convergence result discussed in the last section does not rely on any specific structure of problem \eqref{eq:problem} besides  weak convexity (or the slightly stronger quadratic approximation property in \Cref{assum:prox-linear}). Nonetheless, many applications give rise to instances of problem \eqref{eq:problem} that have additional structures, which can potentially be exploited by our incremental methods to achieve faster convergence rates. In this section, we focus on a structural property called \emph{sharpness} and show that all three incremental methods for solving problem \eqref{eq:problem} will achieve a local linear rate of convergence when the problem possesses the sharpness property.

\subsection{Sharpness: Weak sharp minima}
We start by defining the sharpness property through the notion of weak sharp minima.

\begin{defi}[sharpness; cf. \cite{burke1993weak}]\label{def:sharpness}
Consider problem \eqref{eq:problem}. We say that $\calX\subseteq \calC$ is a set of weak sharp minima for the function $f:\R^n\rightarrow \R$ over $\calC$ with parameter $\alpha>0$ if for any $\vx\in \calC$, we have 
\[ f(\vx ) - f(\vy) \geq \alpha \dist(\vx , \calX )\]
for all $\vy\in \calX$. We say that problem \eqref{eq:problem} possesses the sharpness property if it has a set of weak sharp minima.
\end{defi}

%Besides \Cref{assum 1} (or \Cref{assum 2}), we make the following  additional sharpness assumption on $f$ in \eqref{eq:problem} through out this Section. We have listed in \Cref{sec:weakly convex examples} that  a set of concrete weakly convex functions satisfies the sharpness condition. 
%
%\begin{assum}\label{assum 3}
%	\begin{itemize}
%		\item[(A5)] (sharpness) The function $f$ in \eqref{eq:problem} is $\alpha$-sharp in \Cref{def:sharpness}. 
%	\end{itemize}
%\end{assum}
The notion of sharpness  plays an important role in establishing linear convergence results for subgradient-based methods and superlinear convergence results for proximal-based methods; see, e.g., \cite{davis2018subgradient,goffin1977convergence,nedic2001convergence,Shor:1985:MMN:3585}. 

As it turns out, many concrete applications give rise to sharp instances of problem \eqref{eq:problem}.
For instance, it is shown in \cite{li2018nonconvex} that the robust matrix sensing (RMS) problem \eqref{eq:rms factorization} possesses the sharpness property under certain statistical conditions.  The work \cite{li2018nonconvex} also reveals a more general implication, namely once the measurement operator $\calA$ possesses the so-called $\ell_1/\ell_2$-RIP (see \cite{li2018nonconvex} for the definition), then the associated problem will have the sharpness property. Such an implication opens up the possibility of establishing the sharpness property of a wide class of signal recovery problems, such as robust blind deconvolution \cite{charisopoulos2019composite} and robust phase retrieval problems \cite{charisopoulos2019low,duchi2017solving}.

Now, suppose that problem \eqref{eq:problem} possesses the sharpness property with parameter $\alpha>0$. Consider the case where all the component functions $f_1,\ldots,f_m$ in problem \eqref{eq:problem} satisfy the bounded subgradients assumption in \Cref{assum:subgradient and proximal point} with parameter $L_1$ on an open convex set $\calQ$ that contains $\calC$. Let 
\[ \widetilde{L}_1 :=\sup \left\{   \|\widetilde \nabla f(\vx)\|_2: \vx\in \calQ, \widetilde \nabla f(\vx)\in \partial f(\vx) \right\}. \]
It is clear that $L_1\geq \widetilde{L}_1$.  According to \cite[Theorem 9.13]{rockafellar2009variational}, $f$ is Lipschitz continuous on $\calQ$ with parameter $\widetilde L_1$. This, together with \Cref{def:sharpness} (take $\vy \in \calP_{\calX}(\vx)$), gives
\e\label{eq:alpha L}
\alpha \dist(\vx,\calX) \leq f(\vx) -  f(\vy)\leq \widetilde L_1\dist(\vx,\calX).
\ee
Consequently, we have $\alpha \le \widetilde L_1\le L_1$.

Next, consider the case where all the component functions $f_1,\ldots,f_m$ in problem \eqref{eq:problem} satisfy the bounded subgradients assumption in \Cref{assum:prox-linear} with parameter $L_2$ on an open convex set $\calQ$ that contains $\calC$. Let 
\[ \widetilde{L}_2:= \sup \left\{  \|\widetilde \nabla f(\vx;\overline \vx)\|_2: \vx,\overline \vx  \in \calQ; \widetilde \nabla f(\vx;\overline \vx)\in \partial f(\vx;\overline \vx) \right\}. \]
It is clear that $L_2 \geq \widetilde{L}_2$.  By \cite[Theorem 9.13]{rockafellar2009variational}, $\vx\mapsto f(\vx;\overline \vx)$ is Lipschitz continuous on $\calQ$ with parameter $\widetilde L_2$ for any $\overline \vx \in \calQ$, which gives
\[
   f(\overline \vx) - f(\vx;\overline \vx)  =  f(\overline \vx;\overline \vx) - f(\vx;\overline \vx)  \leq \widetilde L_2 \|\vx-\overline \vx\|_2.
\]
 According to \Cref{assum:prox-linear}, we have
\[
   f(\vx;\overline \vx) - f(\vx) \leq \frac{\tau}{2} \|\vx-\overline \vx\|_2^2.
\]
It follows that for any $\overline \vx \in \calQ$ satisfying $\overline \vx \not= \vx$,
\e\label{eq:lip}
   \frac{f(\overline \vx) - f(\vx)}{\|\vx-\overline \vx\|_2} \leq \widetilde L_2 + \frac{\tau}{2} \|\vx-\overline \vx\|_2.
%   \liminf_{\overline \vx \rightarrow\vx} \frac{f(\overline \vx) - f(\vx)}{\|\vx-\overline \vx\|_2} \leq \liminf_{\overline \vx\rightarrow\vx} \frac{\widetilde L_2 \|\vx-\overline \vx\|_2 + \frac{\tau}{2} \|\vx-\overline \vx\|_2^2 }{\|\vx-\overline \vx\|_2} = \widetilde L_2.
\ee
Since the function $f$ is weakly convex, its subdifferential can be equivalently characterized as
%Note that the subdifferential defined in \eqref{eq:subdifferential} for $\tau$-weakly convex functions is equivalent to the following definition in the sense of Fr\'echet:
\e\label{eq:frechet}
     \partial f(\vx) = \left\{  \widetilde \nabla f(\vx)\in \R^n: \liminf\limits_{\overline \vx \rightarrow \vx} \frac{f(\overline \vx)-f(\vx) - \langle 
     	\widetilde \nabla f(\vx), \overline \vx-\vx \rangle}{\|\overline \vx -\vx\|_2} \geq 0 \right\};
\ee
see, e.g.,~\cite{LSM20}. Upon taking  $\overline \vx = \vx+t \widetilde \nabla f(\vx)$ with $t \searrow 0$ and plugging \eqref{eq:lip} into \eqref{eq:frechet}, we get 
\[
    \|\widetilde \nabla f(\vx)\|_2 \leq \widetilde{ L}_2, \quad \forall\ \vx\in \calQ, \widetilde \nabla f(\vx) \in \partial f(\vx).
\]
By \cite[Theorem 9.13]{rockafellar2009variational}, this implies that $f$ is Lipschitz continuous with parameter $\widetilde L_2$ on $\calQ$. It then follows from~\eqref{eq:alpha L} that $\alpha \le \widetilde L_2 \le L_2$. 

Summarizing the above discussion, since $L = \max\{L_1, L_2\}$, we obtain the following intrinsic relation:
\e\label{eq:L alpha}
  \alpha \leq L.
\ee

\subsection{The key recursion}

The following important recursion, which is shared by all three incremental methods, allows us to exploit the sharpness property in our convergence analysis of these methods.
\begin{prop}[key recursion for linear convergence]\label{prop:local linear convergence}
	Under the setting of \Cref{prop:global convergence}, for any  $\vx^\star \in \calC$, we have
	\e\label{eq:local linear convergence recursion}
	\begin{aligned}
		\|\vx_{k+1} -\vx^\star \|^2_2  &\leq  (1+2m\tau \mu_k )\|\vx_{k} -\vx^\star \|^2_2  - 2m\mu_{k}(f(\vx_{k}) - f(\vx^\star)) \\
		&\quad+   2m^2\mu_{k}^2 L^2 + 2\tau m^3  \mu_k^3 L^2.
	\end{aligned}
	\ee
\end{prop}

\begin{proof}%[Proof of \Cref{prop:local linear convergence}]
	Letting $\vy = \vx^\star$ in \Cref{lem:basic recursion} and summing~\eqref{eq:preliminary recursion} over $i=1,\ldots,m$ give
	\e\label{eq:recursion 3}
	\begin{aligned}
		\|\vx_{k+1} -\vx^\star \|^2_2 &\leq  \|\vx_{k} -\vx^\star \|^2_2   - 2\mu_k   \underbrace{\sum_{i=1}^{m} (f_i(\vz_i)-f_i(\vx^\star)) }_{\Delta_1}  \\
		&\quad +  \tau  \mu_k   \underbrace{\sum_{i=1}^{m}\|\vz_i - \vx^\star\|^2_2 }_{\Delta_2} + \underbrace{ \gamma \sum_{i = 1}^m \|\vx_{k,i}-\vx_{k,i-1}\|_2^\theta}_{\Delta_3}.
		%		&=  \|\vx_{k} -\vx^\star \|^2_2 - 2\mu_k \sum_{i=1}^{m}(f_i(\vx_{k,i-1})- f_i(\vx_k) + f_i(\vx_k) - f_i(\vx^\star)) \\
		%		&\quad + \tau  \mu_k  \sum_{i=1}^{m}\|\vx_{k,i-1} - \vx_k + \vx_{k} - \vx^\star\|^2_2  + m\mu_{k}^2 L^2.
	\end{aligned}
	\ee
	
 We now derive bounds for $\Delta_1$, $\Delta_2$, and $\Delta_3$ in \eqref{eq:recursion 3} for different incremental methods. 
	
	\vspace{0.2cm}
	\emph{Part I:  Incremental subgradient method}. Recall that $\vz_i = \vx_{k,i-1}$, $\gamma = 1$, $\theta = 2$ in \eqref{eq:recursion 3} for this case.  To bound $\Delta_1$ and $\Delta_2$, following the derivations in \eqref{eq:Delta 1}--\eqref{eq:Delta 3 v1},  we have 
	\e\label{eq:Delta 1 linear rate}
	\begin{split}
		\Delta_1 &=   \sum_{i=1}^{m} \left( f_i(\vx_{k,i-1})- f_i(\vx_k) + f_i(\vx_k) - f_i(\vx^\star) \right)  \\
		&\geq  m \big(f(\vx_{k}) - f(\vx^\star) \big)-   \frac{m^2-m}{2}   \mu_k L^2,
	\end{split}
	\ee
	\e\label{eq:Delta 2 prop linear rate}
	\begin{split}
		\Delta_2 &=   \sum_{i=1}^{m}\|\vx_{k,i-1} - \vx_k + \vx_{k} - \vx^\star\|^2_2 \\
		&\leq 2m  \| \vx_{k} - \vx^\star\|^2_2 +   \frac{(m-1)m(2m-1)}{3} \mu_k^2  L^2,
	\end{split}
	\ee	
	and 
	\e\label{eq:Delta 3 linear rate}
	\Delta_3 \leq  m \mu_{k}^2 L^2.
	\ee
	Substituting \eqref{eq:Delta 1 linear rate}--\eqref{eq:Delta 3 linear rate} into  \eqref{eq:recursion 3}  yields
	\e\label{eq:recursion 3-4}
	\begin{aligned}
		\|\vx_{k+1} -\vx^\star \|^2_2 & \leq  (1+2m\tau \mu_k )\|\vx_{k} -\vx^\star \|^2_2  - 2m\mu_{k}(f(\vx_{k}) - f(\vx^\star))\\
		&\quad+   2m^2\mu_{k}^2 L^2 + 2 \tau m^3  \mu_k^3 L^2.
	\end{aligned}
	\ee
	%	For all $k\geq 0$, one can now take $\vx^\star $ as the projection of $\vx_{k}$ onto the optimal solution set $\calX$. This, together with the fact that    $\dist(\vx_{k+1},\calX) \leq \|\vx_{k+1} -\vx^\star \|_2$ and the sharpness property of $f$ in \Cref{assum 3}, implies 
	%	\e \label{eq:basic recursion 1}
	%	\begin{aligned}
	%		\dist^2(\vx_{k+1},\calX)  &\leq (1+2m\tau \mu_k ) \dist^2(\vx_{k},\calX)   - 2m\alpha\mu_{k} \dist(\vx_k,\calX)
	%		+   m^2\mu_{k}^2 L^2 + \tau m^3  \mu_k^3 L^2,
	%	\end{aligned}
	%	\ee 
	Note that we have enlarged $m^2\mu_{k}^2 L^2$ to $2m^2\mu_{k}^2 L^2$ and $\frac{(m-1)m(2m-1)}{3}$ to $2m^3$  in order for  all  three incremental methods to satisfy this recursion.

	\vspace{0.2cm}
	\emph{Part II: Incremental proximal point method}. Recall that $\vz_i = \vx_{k,i}$ and $\gamma = 0$ (i.e., $\Delta_3=0$) in \eqref{eq:recursion 3} for this case. Following the analysis in \eqref{eq:Delta 1-2}--\eqref{eq:Delta 2-2}, we have 
	\e\label{eq:Delta 1 and 2}
	\begin{aligned}
		\Delta_1 & \geq m \big(f(\vx_k) - f(\vx^\star)\big) -  \frac{m^2+m}{2}   \mu_k L^2, \\
		\Delta_2 &\leq   2 m\| \vx_k - \vx^\star \|_2^2 + \frac{m(m+1)(2m+1)}{3} \mu_k^2  L^2.
	\end{aligned}
	\ee
	Substituting  \eqref{eq:Delta 1 and 2} into \eqref{eq:recursion 3} yields the same recursion as \eqref{eq:recursion 3-4}, where we have enlarged $m^2+m$ to $2m^2$ and $\frac{m(m+1)(2m+1)}{3}$ to $2m^3$.

	\vspace{0.2cm}
	\emph{Part III: Incremental prox-linear method}. 
	Recall that $\vz_i = \vx_{k,i-1}$, $\gamma = 2\mu_kL$, $\theta = 1$ in \eqref{eq:recursion 3} for this case. Thus, the bounds for $\Delta_1$ and $\Delta_2$ will be the same as those in \eqref{eq:Delta 1 linear rate} and \eqref{eq:Delta 2 prop linear rate}, respectively. In addition, we have  $\Delta_3 \leq  2m \mu_{k}^2 L^2$ from \eqref{eq:inner iteration length}. Substituting the bounds for $\Delta_1$, $\Delta_2$, and $\Delta_3$ into  \eqref{eq:recursion 3} yields the same recursion as \eqref{eq:recursion 3-4}.
\end{proof}

\subsection{Linear convergence result}
It is known that subgradient-based methods with a constant stepsize for  nonsmooth optimization  may  never converge to the  optimal solution set.   In order to get exact convergence,  diminishing stepsizes are generally needed \cite{goffin1977convergence,Shor:1985:MMN:3585}.  The work \cite{nedic2001convergence} analyzes the incremental subgradient method using  Polyak's  dynamic stepsizes $\mu_k = (f(\vx_k) - f^\star)/\|\widetilde{\nabla} f(\vx_k)\|^2_2$ for some $\widetilde{\nabla} f(\vx_k) \in  \partial f(\vx_k)$ with  $\widetilde \nabla f (\vx_{k}) \neq \mathbf{0}$, which can be regarded as an adaptive diminishing stepsize rule. However, the implementation of  Polyak's  dynamic stepsizes  requires the knowledge of the optimal function value $f^\star$,  which makes it impractical. Motivated by previous works on full subgradient methods \cite{davis2018subgradient,goffin1977convergence,Shor:1985:MMN:3585}, in this section, we show that by using  geometrically diminishing stepsizes of the form $\mu_k = \rho^k \mu_0$, $k=0,1,\ldots$,  all  three incremental methods  have a local linear rate of convergence. Specifically, we have the following theorem.

\begin{thm}\label{thm:local linear convergence}
	Consider the setting of \Cref{prop:global convergence}.  Suppose further that $\calX$ is a set of weak sharp minima for problem \eqref{eq:problem} (see \Cref{def:sharpness}). 
	Let $\{\vx_k\}_{k\geq 0}$ be the sequence  generated by any of the three incremental methods for solving problem \eqref{eq:problem},  where the initial point $\vx_{0}$ satisfies
	\[\dist(\vx_0, \calX) <  \frac{\alpha}{\tau}\] and the stepsizes are chosen as
	\[
	\mu_{k} = \rho^k\cdot \mu_0
	\]
	with
	\[ \tau \mu_0 <   \frac{\alpha^2}{6m L^2}\left(1- \left(\max\left\{\frac{2\tau}{\alpha} \dist(\vx_0,\calX) - 1, 0\right\}\right)^2  \right),  \]
	\[  	 \rho \in [\rho_{\min}, 1), \quad \quad \rho_{\min}:= \sqrt{1 +  2m\left( \tau - \frac{\alpha}{ e_0}\right)\mu_{0}  + \frac{3m^2 L^2 }{ e_0^2}\mu_0^2},  \]
	and
	\[ e_0 := \max \left\{ \dist(\vx_0, \calX), \frac{3mL^2 \mu_0}{\alpha}\right\}.   \]
	Then, we have 
	\e\label{eq:linear rate}
	\dist(\vx_k, \calX) \leq \rho^k \cdot e_0, \quad \forall \ k\geq 0.
	\ee
\end{thm}

\begin{proof}%[Proof of \Cref{thm:local linear convergence}]
%	To start, we  verify that the upper bound on $\mu_0$  is well defined.  This consists of two folds: 1) showing it does not contradict with the definition of $e_0$  and 2) showing that this upper bound is positive.  For the first part, we have $\mu_0\leq \frac{\alpha e_0}{2mL^2}$, which   We first consider the case  $\tau=0$ (i.e., convex case), where the upper bound on  $\mu_0$ is clearly positive and $\mu_0\leq \frac{4\alpha e_0}{5mL^2}$ is ensured by the definition of $e_0$. We now consider the case  $\tau>0$. Note that the upper bound on $\mu_0$ is a quadratic function with respect to $e_0$, we have that $\mu_0\leq \frac{\alpha^2}{5m\tau L^2}$ always holds true by taking its maximum over $e_0$.    
%	
	 
	We first  show that $\rho_{\min}\in(0,1)$, so that $\rho$ is well defined. Note that $  \rho_{\min} = \sqrt{1+v(\mu_0)}$ with $
	v(\mu_0) = 2m\left( \tau - \frac{\alpha}{ e_0}\right)\mu_{0}  + \frac{3m^2 L^2 }{ e_0^2}\mu_0^2$
	being a quadratic function of $\mu_0$. To establish $\rho_{\min}\in(0,1)$, it suffices to show that  $v(\mu_0)\in(-1,0)$.    On one hand, we have $v(\mu_0) > - \frac{2m\alpha \mu_{0}}{ e_0} \geq -\frac{2\alpha^2}{3L^2} > -1$  by the definition of $e_0$ and the fact that $L\geq \alpha$ (see \eqref{eq:L alpha}).  On the other hand, we prove that $v(\mu_0)< 0$, or equivalently,
	\e\label{eq:rho min 0}
	   \mu_0<\frac{2(\alpha e_0 - \tau e_0^2)}{3mL^2}.
	\ee 
    Towards establishing \eqref{eq:rho min 0}, we consider the cases $\tau=0$  and $\tau>0$ separately.  In the case  where $\tau=0$ (i.e., $f$ is convex),  \eqref{eq:rho min 0} can be ensured by the definition of $e_0$.  In the case where $\tau>0$, since  $\mu_0< \frac{\alpha^2}{  6m\tau  L^2}$, we have $\frac{3mL^2 \mu_0}{\alpha} < \frac{\alpha}{2\tau}$.  If $\dist(\vx_{0},\calX) \geq \frac{\alpha}{2\tau}$, then $e_0 = \dist(\vx_{0},\calX) \geq \frac{\alpha}{2\tau}$ and the inequality defining $\mu_0$ reduces to 
    \[
    \mu_0<\frac{2(\alpha\dist(\vx_{0},\calX) - \tau \dist(\vx_{0},\calX)^2 )}{3mL^2} = \frac{2(\alpha e_0 - \tau e_0^2 )}{3mL^2},
    \]
    which is exactly \eqref{eq:rho min 0}.  If $\dist(\vx_{0},\calX) < \frac{\alpha}{2\tau}$,  then $e_0 <   \frac{\alpha}{2\tau}$,  $\mu_0<\frac{\alpha^2}{  6m\tau  L^2}$, and  \eqref{eq:rho min 0} is equivalent to 
	\[ %\e\label{eq:rho min 1}
	    2\tau e_0^2 - 2\alpha e_0 + 3m L^2 \mu_0 < 0.
	\] %\ee
	The above is a quadratic inequality in $e_0$, which is further equivalent to 
	\e\label{eq:rho min 2}
	      \frac{\alpha - \sqrt{\alpha^2 - 6m\tau L^2\mu_0}}{2\tau} < e_0 < \frac{\alpha + \sqrt{\alpha^2 - 6m\tau L^2\mu_0}}{2\tau}.
	\ee
	The second inequality in \eqref{eq:rho min 2} is ensured by $e_0< \frac{\alpha}{2\tau}$, while the first inequality follows from the definition of $e_0$ and the fact that $\frac{\alpha - \sqrt{\alpha^2 - 6m\tau L^2\mu_0}}{2\tau} = \frac{ 6m\tau L^2\mu_0}{2\tau(\alpha + \sqrt{\alpha^2 - 6m\tau L^2\mu_0})} < \frac{3mL^2\mu_{0}}{\alpha} \leq e_0$.  Hence, we have proved that $\rho_{\min}\in(0,1)$. 
	
	We now prove \eqref{eq:linear rate} by induction. It is trivially true when $k=0$ due to the definition of $e_0$. Assuming that $\dist(\vx_k, \calX) \leq \rho^k \cdot e_0$ for some $k\geq 0$, it remains to show that $\dist(\vx_{k+1}, \calX) \leq \rho^{k+1} \cdot e_0$.  
	
	By applying \Cref{prop:local linear convergence} with $\bm{x}^\star = \mathcal{P}_{\mathcal{X}} ( \bm{x}_k )$, using the fact that    $\dist(\vx_{k+1},\calX) \leq \|\vx_{k+1} -\vx^\star \|_2$ and $\rho<1$, and utilizing the sharpness property, we have 
	\e \label{eq:basic recursion}
	\begin{aligned}
		\dist^2(\vx_{k+1},\calX)  &\leq  (1+ 2m\tau \mu_0 ) \dist^2(\vx_{k},\calX) -  2m\alpha\mu_{k} \dist(\vx_k,\calX) \\
		&\quad+   2m^2\mu_{k}^2 L^2 + 2\tau m^3  \mu_k^3 L^2.
	\end{aligned}
	\ee 
	Observe that the right-hand side of \eqref{eq:basic recursion} is a  quadratic function of $\dist(\vx_k,\calX)$. Since $0\leq \dist(\vx_k,\calX) \leq \rho^k \cdot e_0$, it follows that the right-hand side of \eqref{eq:basic recursion} achieves its maximum at $\dist(\vx_{k},\calX) = \rho^k \cdot e_0$, as one has  $\frac{2m\alpha \mu_0}{1+2m\tau\mu_0} < 2m\alpha \mu_0 < e_0$ by the definition of $e_0$ and the fact that $L\geq \alpha$ (see \eqref{eq:L alpha}). Upon plugging $\dist(\vx_{k},\calX) = \rho^k \cdot e_0$ and $\mu_k = \rho^k \mu_0$ into \eqref{eq:basic recursion}, we obtain
	\[ %\e\label{eq:linear convergence 1}
	\begin{aligned}
		\dist^2(\vx_{k+1},\calX)  &\leq \rho^{2k}  e_0^2 + 2m\tau \mu_0 \rho^{2k}  e_0^2 - 2m\alpha \mu_0 \rho^{2k} e_0\\
		&\quad  + 2m^2\mu_{0}^2 L^2\rho^{2k} + 2\tau m^3  \mu_0^3 L^2\rho^{2k}\\
		& = \rho^{2k}  e_0^2 \left( 1 + 2m \left( \tau - \frac{\alpha}{ e_0}\right)\mu_{0}  + \frac{2m^2 L^2 + 2\tau m^3  \mu_0 L^2 }{ e_0^2}\mu_0^2  \right)  \\
		&< \rho^{2k}  e_0^2 \left( 1 +  2m\left( \tau - \frac{\alpha}{ e_0}\right)\mu_{0}  + \frac{3m^2 L^2 }{ e_0^2}\mu_0^2  \right) \\
		& = \rho^{2k}  e_0^2 \cdot  \rho_{\min}^2\\
		&\leq \rho^{2k+2}  e_0^2,
	\end{aligned}	     
	\] %\ee
	where in the first inequality we use $\rho <1$ and in the second inequality we use $\tau\mu_0 < \frac{1}{6m}$ by \eqref{eq:L alpha} and the definition of $\mu_0$.  This completes the inductive step and hence the proof of the theorem.
\end{proof}	

Before we leave this section, let us make two remarks. First, \Cref{thm:local linear convergence} provides a unified analysis of the incremental methods in both the convex  (i.e., $\tau =0$) and nonconvex  (i.e., $\tau>0$) cases. In the convex case, the initial distance $\dist(\vx_0, \calX)$ and initial stepsize $\mu_0$ can be chosen arbitrarily.   
%The theorem   asserts that if any of the three incremental methods is initialized properly, then with  geometrically diminishing stepsizes of the form $\mu_{k} = \rho^k \mu_{0}$, the algorithm will generate a  sequence of iterates which converges to the set of weak sharp minima $\calX$  \emph{linearly}.  
The  diminishing stepsize rule in \Cref{thm:local linear convergence} is practical in the sense of implementation. The initial stepsize $\mu_0$ and decay factor $\rho$ can be computed  with  the knowledge of  the problem parameters $\alpha,\tau,L$ and the initial distance.  Even in the case where it is difficult to estimate these  parameters accurately, one can still obtain  proper $\mu_0$ and $\rho$ for specific applications.    This is in sharp contrast to  Polyak's dynamic stepsize rule, which requires the knowledge of the optimal function value $f^\star$. Such information is often hard to obtain in practice.

Second, we note that the three incremental methods may have different performance in practice, although \Cref{thm:local linear convergence} has the same convergence guarantee for them. In other words, \Cref{thm:local linear convergence} only provides sufficient (but possibly not necessary) conditions for the methods to converge and \eqref{eq:linear rate} only provides an upper bound on the convergence rate. In the next section, we show the (different) performance of the three incremental methods for solving the RMS problem \eqref{eq:rms factorization}. As will be seen,  the incremental prox-linear methods often performs  best. 

\section{Experiments}
\label{sec:experiments}
In this section, we conduct a series of experiments on the three incremental methods for solving the RMS problem  \eqref{eq:rms factorization}.\footnote{Our code is available at \url{https://github.com/lixiao0982/Incremental-methods}}
We list below the abbreviations of all the algorithms that appear in this section.

\vspace{0.3cm}
\begin{mdframed}[backgroundcolor=black!3,rightline=false,leftline=false,    leftmargin =0pt,
	rightmargin =0pt,
	innerleftmargin =2pt,
	innerrightmargin =2pt]
	{\small
		\begin{tabular}{ll}
			SGM: & Full subgradient method \\
			SGD:  & Stochastic subgradient method\\
			SPL:  & Stochastic prox-linear method\\
			ISG:  & Incremental subgradient method\\
			IPP:  & Incremental proximal point method \\
			IPL:  & Incremental prox-linear method\\
		\end{tabular}
	}	
\end{mdframed}
%\vspace{0.2cm}
%In IPP algorithm, we have to solve a subproblem in each inner iteration. We utilize the MinFunc \cite{schmidt2005minfunc} solver to tackle this subproblem. 

%\paragraph{Robust Matrix Sensing (RMS)}
We generate $\mU^\star\in\R^{n\times r}$ and the $m$ sensing matrices $\mA_1,\ldots,\mA_m\in\R^{n\times n}$ (which defines the linear operator $\calA$)  with i.i.d.  standard Gaussian entries.  The ground-truth  low-rank matrix is then generated by $\mX^\star = \mU^\star\mU^{\star \T}$. We generate the outlier vector $\vs^\star\in\R^m$ by first randomly selecting $pm$ locations, where $p$ is the  ratio of outliers. Then,   each of the selected location is filled with an i.i.d. mean 0 and variance 10 Gaussian, while  the remaining locations are set to 0.   According to \eqref{eq:rms model}, the linear measurement $\vy \in \R^m$ is generated by     $y_i = \langle \mA_i, \mX^\star\rangle + s^\star_i$,  $ i=1,\ldots,m$.   We set  $n  = 50, r = 5, m = 5nr, p = 0.3$. As stated in \cite[Proposition 2]{li2018nonconvex}, the RMS problem possesses the sharpness property (see \Cref{def:sharpness}) under this setting. Furthermore, its  optimal solution set is precisely $\calU = \{\mU^\star \mR: \mR\in \R^{r\times r},\mR\mR^T = \mId\} $.
The bounded subgradients assumption of the RMS problem \eqref{eq:rms factorization} on any bounded subset can be verified using \cite[Propsition 4]{li2018nonconvex}. We have now checked   the assumptions   in our \Cref{thm:local linear convergence} (i.e., weak convexity, sharpness, and  bounded subgradients) for the RMS problem. 
According to \Cref{thm:local linear convergence}, in order to achieve linear convergence  for the three incremental methods, we need to utilize the geometrically diminishing stepsizes
$\mu_{k} = \rho^k \mu_0$, $k\geq 0$.
For a fair comparison, all the algorithms are initialized with the same point, whose entries are i.i.d. standard Gaussians (i.e., random initialization). It is  worth mentioning that though \Cref{thm:local linear convergence} requires a special initialization, the random initialization works as good as  carefully designed ones (see, e.g., \cite[Theorem 4]{li2018nonconvex} for an initialization strategy) in our experiments.

Recall from \Cref{thm:local linear convergence}  that the three incremental methods converge linearly at the rate of $\calO(\rho^k)$, where  $\rho$ is exactly the decay factor of the geometrically diminishing stepsizes.  Thus, the smaller decay factor $\rho$ an algorithm can choose, the faster it converges. Based on this observation, we evaluate an algorithm by numerically testing the smallest possible $\rho$ it can work with.

\subsection{Efficiency of incremental methods}
We first compare the performance of the three incremental methods studied in this paper. The result is shown in \Cref{fig:compare in. alg}.  We run all the incremental methods for 500 iterations, where the initial stepsize  $\mu_0$ and the decay factor $\rho$  of the geometrically diminishing stepsizes  are selected from $\{1/m, 30/m, 60/m, \ldots, 240/m \}$  and  $\{0.65, 0.7, 0.75, \ldots, 0.95, 0.99\}$, respectively.  For each pair of parameters  $(\rho,\mu_0)$, we average the distances to the optimal solution set of the last 5 iterates. The assessed  algorithm is deemed successful if this averaged distance is no more than $10^{-8}$ and failed otherwise.  In  \Cref{fig:RMS IGD,fig:RMS IPL,fig:RMS IPP}, each plot represents the success probability of 25 independent trials for different pairs $(\rho,\mu_0)$.  The block with white color indicates that the algorithm is $100\%$ successful under the corresponding parameter settings,  while the block with black color implies that all  25 runs  failed. In between, whiter color implies a higher successful rate. We can observe from \Cref{fig:RMS IGD,fig:RMS IPL,fig:RMS IPP} that the three incremental methods have nearly the same lower bound for the decay factor $\rho$ (i.e., $\rho=0.8$) when solving the RMS problem.  Nonetheless, IPL and IPP are much more robust with respect to the selection of the initial stepsize $\mu_0$. In our test, both IPL and IPP  work even with a very large $\mu_0$, say $\mu_0 = {10^4}/{m}$.  Such a robustness property  is crucial in practice, as it allows much less  hand-tuning.

%Recall that our \Cref{thm:local linear convergence} states that the sequence of distance $\{\dist(\mU_k,\calU)\}_{k\geq 0}$ converges linearly at the rate of $\calO(\rho^k)$, where  $\rho$ is exactly the stepsizes decay factor.  Thus, the smaller decay factor $\rho$ an algorithm can choose, the faster it converges. 

To have a more detailed interpretation of  \Cref{fig:RMS IGD,fig:RMS IPL,fig:RMS IPP}, we  show in \Cref{fig:IGD convergence}  the convergence progress of ISG for several representative pairs of parameters  $(\rho,\mu_0)$ used in  \Cref{fig:RMS IGD} (the interpretation of \Cref{fig:RMS IPL,fig:RMS IPP} will be similar).  From the convergence plot, we can draw the following conclusions: 1) If we set $\rho = 0.65$ (say $\rho=0.65$ and $\mu_0 = 1/m$), then ISG stops progressing in a few iterations.   2) If we set $\rho = 0.99$ and choose  $\mu_0$  carefully (say $\rho=0.99$ and $\mu_0 = 1/m$), then ISG converges very slowly and drives the distance to  below $10^{-8}$ after nearly 5000 iterations.  3) If we set $\rho \in[0.8, 0.95]$ and choose $\mu_0$  appropriately (say $\rho = 0.8$ and $\mu_0 = 30/m$ or $\rho = 0.9$ and $\mu_0 = 1/m$), then ISG drives the distance to  below $10^{-8}$ within 500 iterations. We can also  observe that   smaller  $\rho$  leads to  faster convergence of ISG, which  corroborates  the results in \Cref{thm:local linear convergence}.  4) If we set $\rho \in [0.8,  0.99]$ but choose a large $\mu_0$ (say $\rho=0.95$ and $\mu_0 = 60/m$), IGD diverges.  It is interesting to mention  that in our test,  IPL and IPP will not diverge even with a large  $\mu_0$ (say $\rho \in [0.8,  0.99]$ and $\mu_0\in [1/m,10^4/m]$). This  indicates that IPL and IPP are much more robust  than ISG with respect to the choice of the initial stepsize $\mu_0$.

\begin{figure}[!htp]
	\begin{subfigure}{0.49\linewidth}
		\centerline{
			\includegraphics[width=1\textwidth]{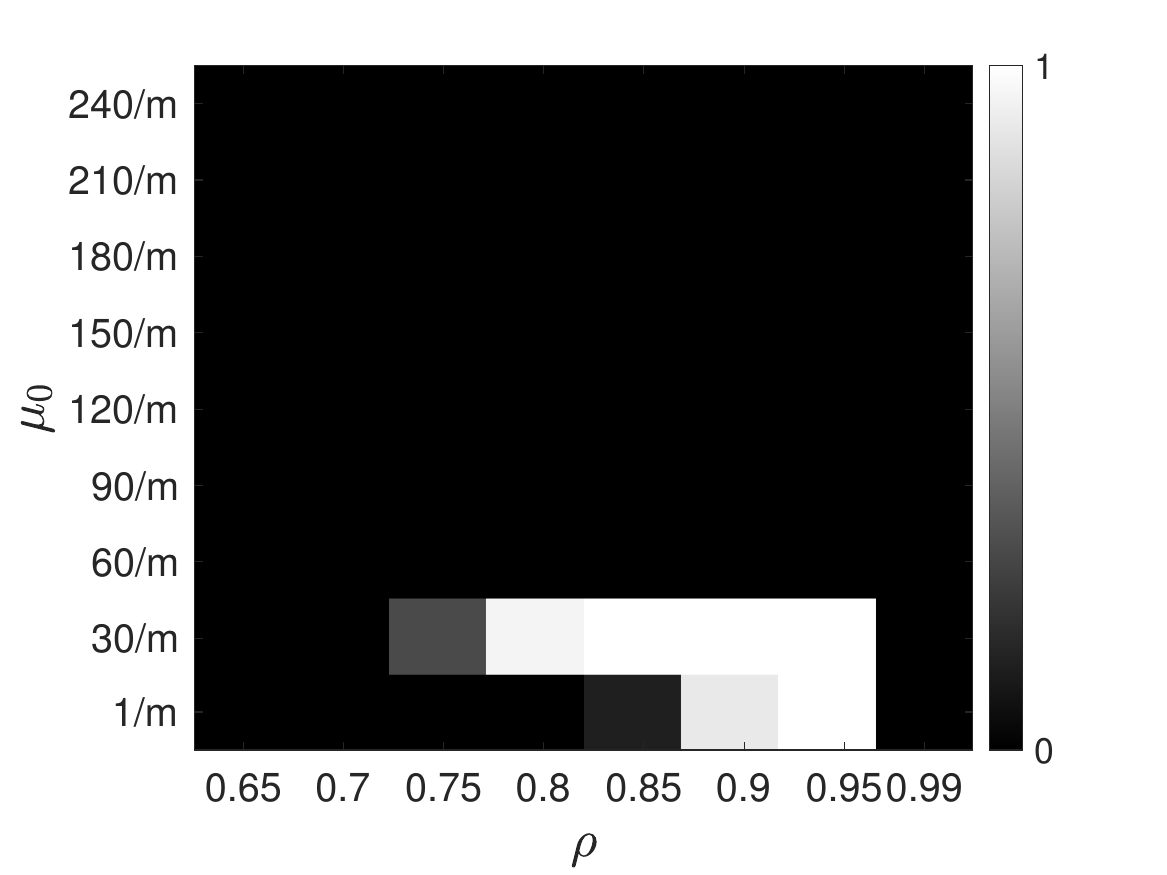}}
		\caption{ISG \label{fig:RMS IGD}}
	\end{subfigure}
	\begin{subfigure}{0.49\linewidth}
		\centerline{
			\includegraphics[width=1\textwidth]{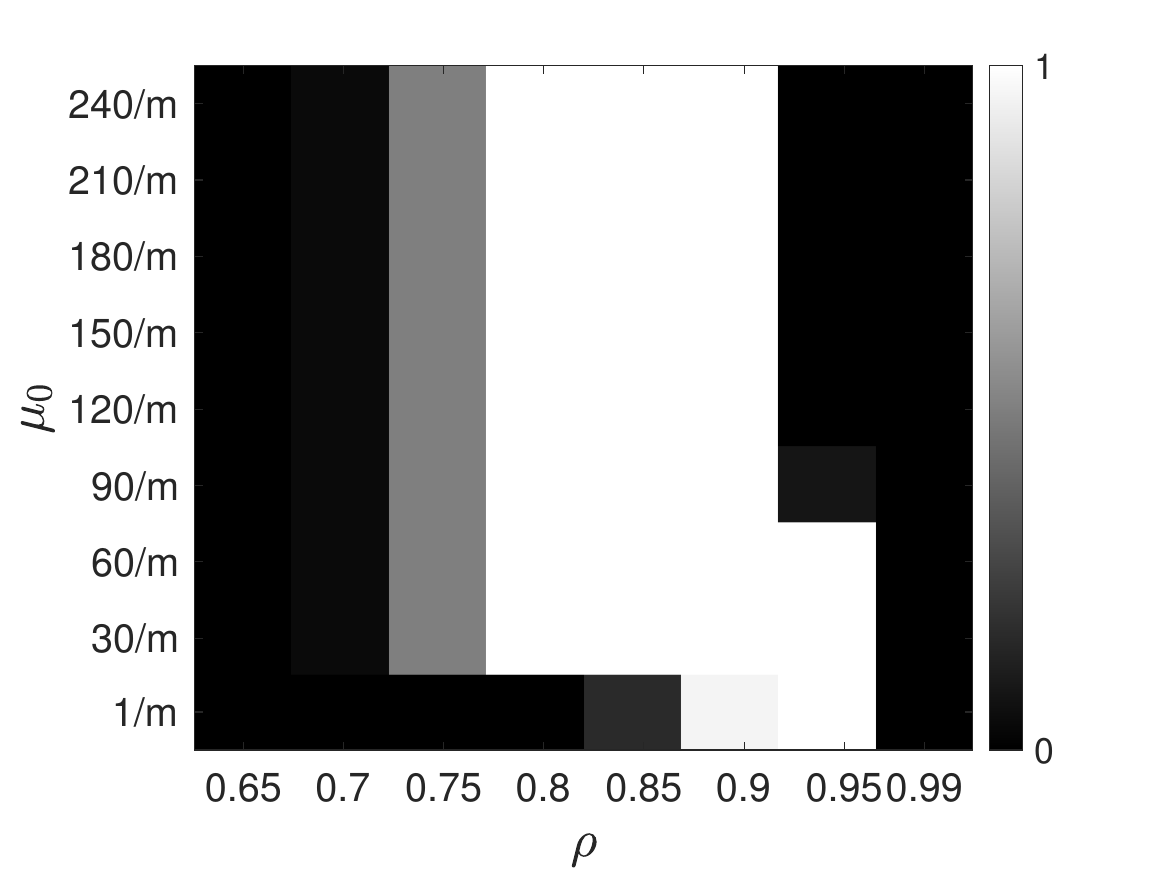}}
		\caption{IPL \label{fig:RMS IPL}}
	\end{subfigure}
\vfill
	\begin{subfigure}{0.49\linewidth}
		\centerline{
			\includegraphics[width=1\textwidth]{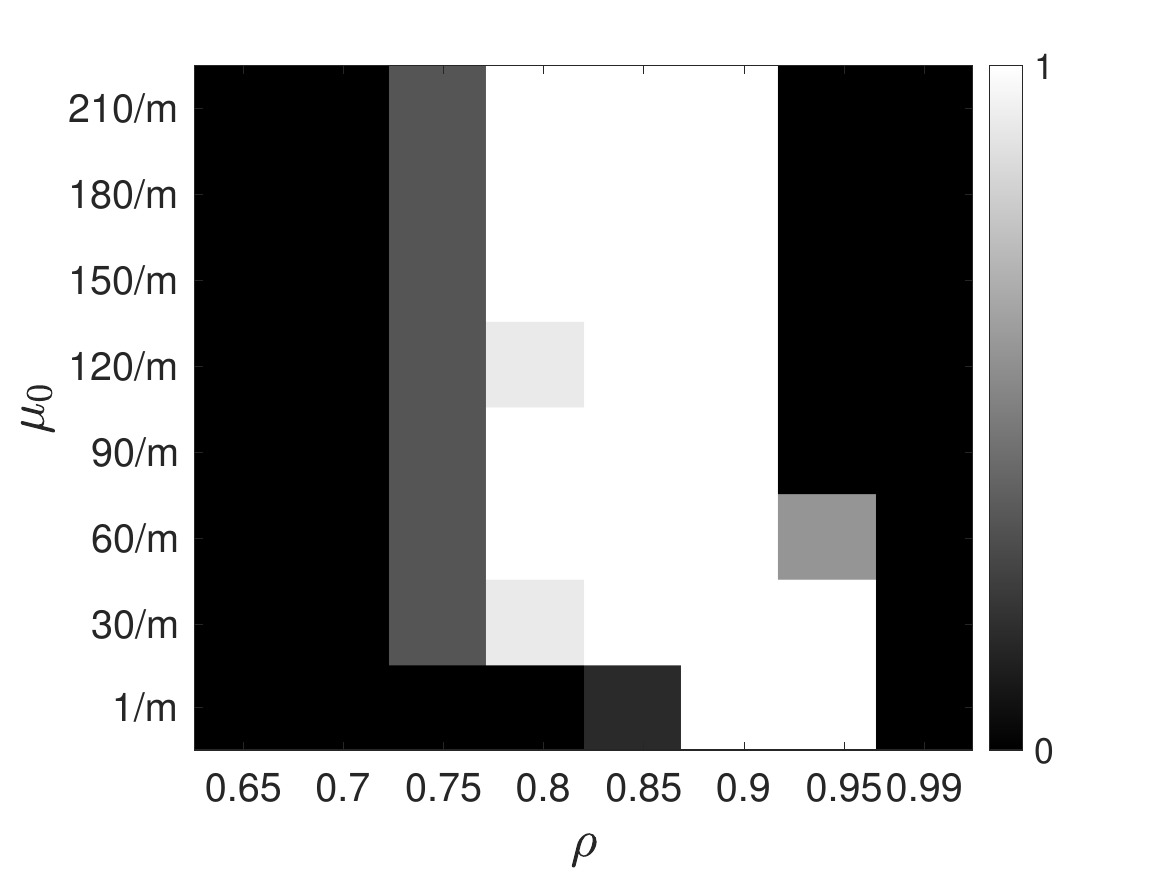}}
		\caption{IPP \label{fig:RMS IPP}}
	\end{subfigure}
	\begin{subfigure}{0.49\linewidth}
		\centerline{
			\includegraphics[width=1\textwidth]{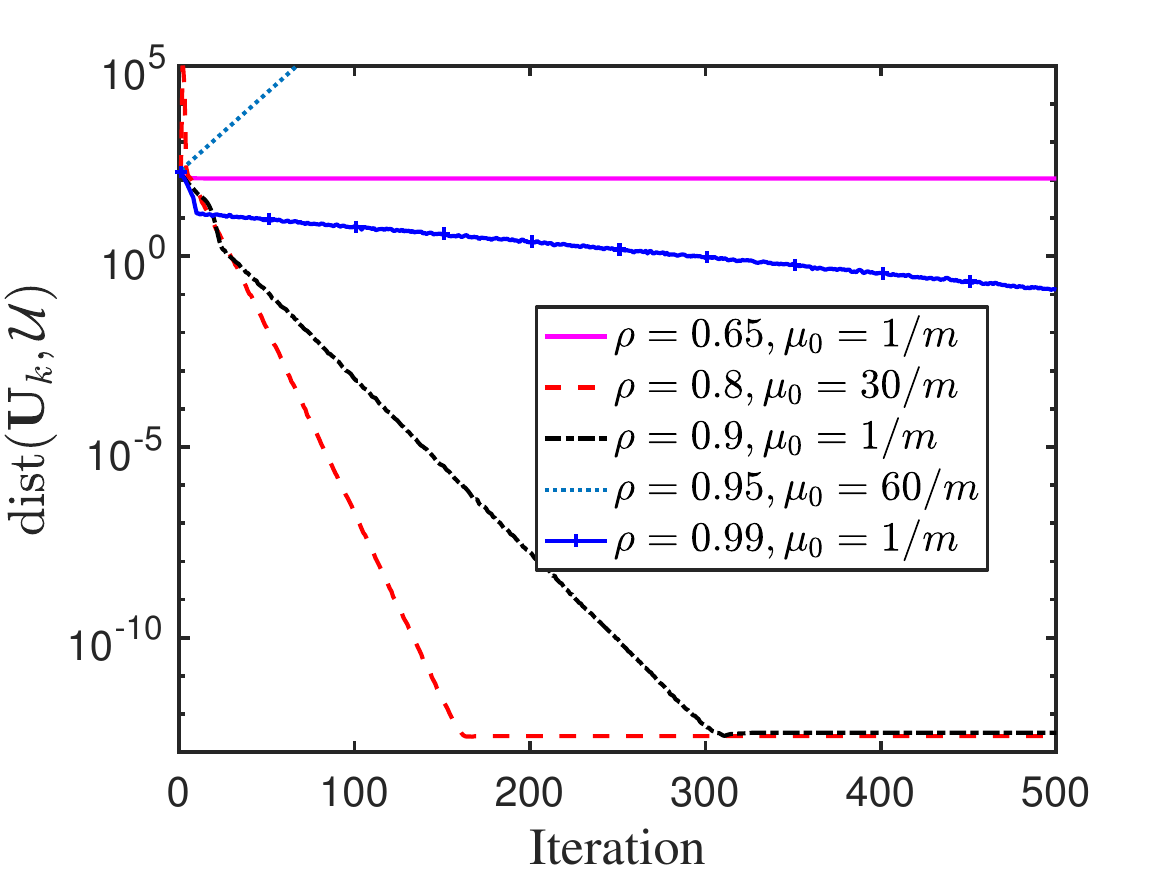}}
		\caption{ISG: convergence plots  \label{fig:IGD convergence}}
	\end{subfigure}
%	\vfill
%	\begin{subfigure}{0.24\linewidth}
%		\centerline{
%			\includegraphics[width=1.5in]{fig/RPR_IGD_algparam}}
%		\caption{RPR: IGD \label{fig:RPR IGD}}
%	\end{subfigure}
%	\begin{subfigure}{0.24\linewidth}
%		\centerline{
%			\includegraphics[width=1.5in]{fig/RPR_IPL_algparam}}
%		\caption{RPR: IPL \label{fig:RPR IPL}}
%	\end{subfigure}
%	\begin{subfigure}{0.24\linewidth}
%		\centerline{
%			\includegraphics[width=1.5in]{fig/RPR_IPP_algparam}}
%		\caption{RPR: IPP \label{fig:RPR IPP}}
%	\end{subfigure}
%	\begin{subfigure}{0.24\linewidth}
%		\centerline{
%			\includegraphics[width=1.5in]{fig/RPR_IPL_showconvergence}}
%		\caption{Convergence  \label{fig:IPL convergence}}
%	\end{subfigure}
	\caption{\small (a)-(c): Performance of the three incremental methods for solving the RMS problem by varying $\rho$ and $\mu_0$. (d) The convergence progress of ISG for several representative pairs of parameters $(\rho,\mu_0)$.
	}\label{fig:compare in. alg}
\end{figure}

\begin{figure}[!htp]
	\begin{subfigure}{0.32\linewidth}
		\centerline{
			\includegraphics[width=1\textwidth]{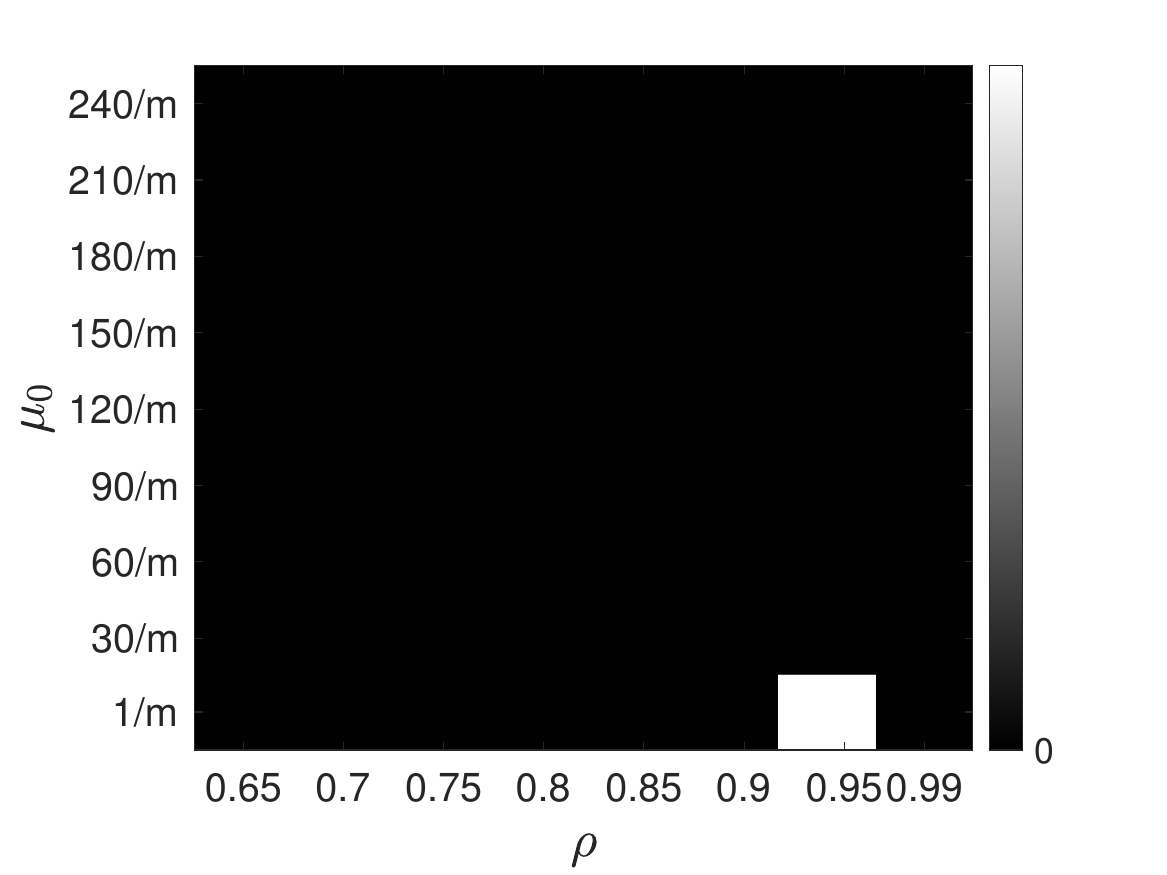}}
		\caption{SGM \label{fig:RMS SGM}}
	\end{subfigure}
	\begin{subfigure}{0.32\linewidth}
		\centerline{
			\includegraphics[width=1\textwidth]{fig/RMS_IGD_algparam}}
		\caption{ISG \label{fig:RMS IGD 2}}
	\end{subfigure}
	\begin{subfigure}{0.32\linewidth}
		\centerline{
			\includegraphics[width=1\textwidth]{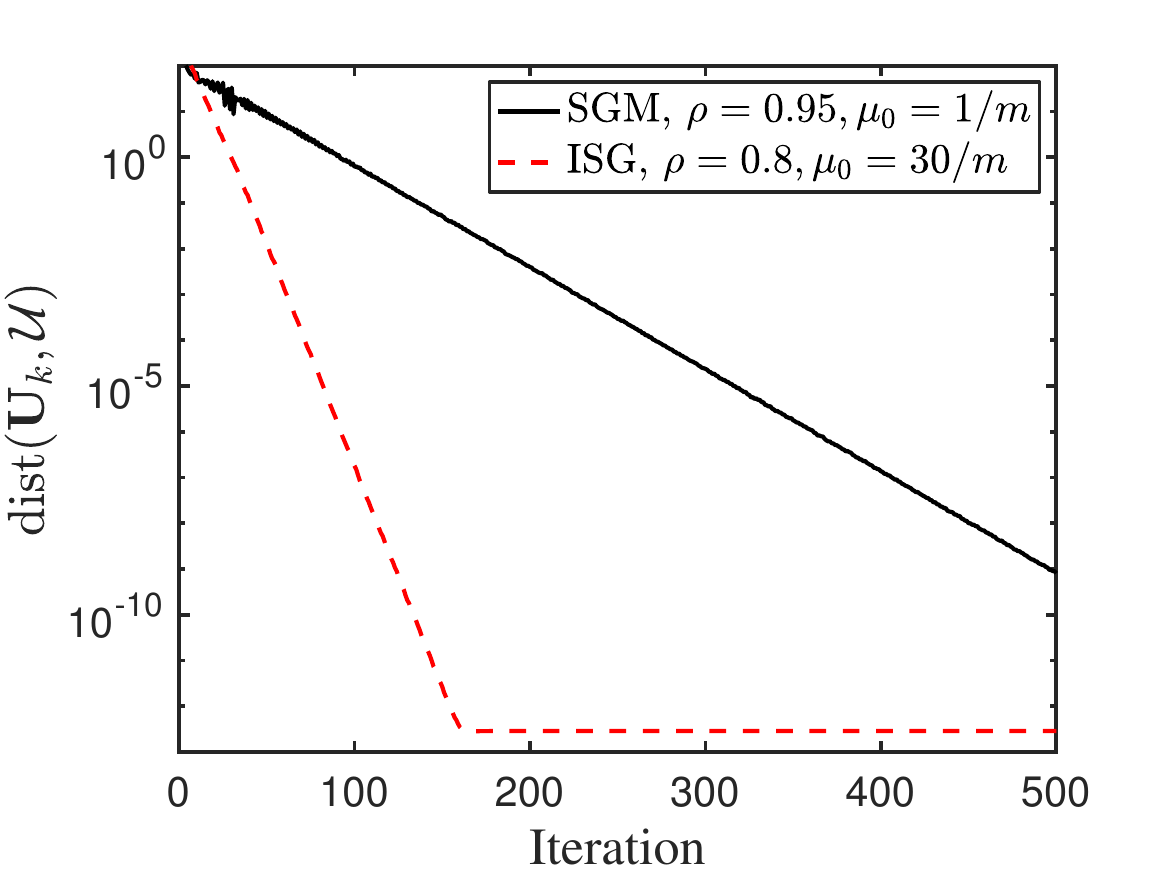}}
		\caption{ISG vs. SGM \label{fig:RMS IGD vs SGM}}
	\end{subfigure}
	\vfill
	\begin{subfigure}{0.32\linewidth}
		\centerline{
			\includegraphics[width=1\textwidth]{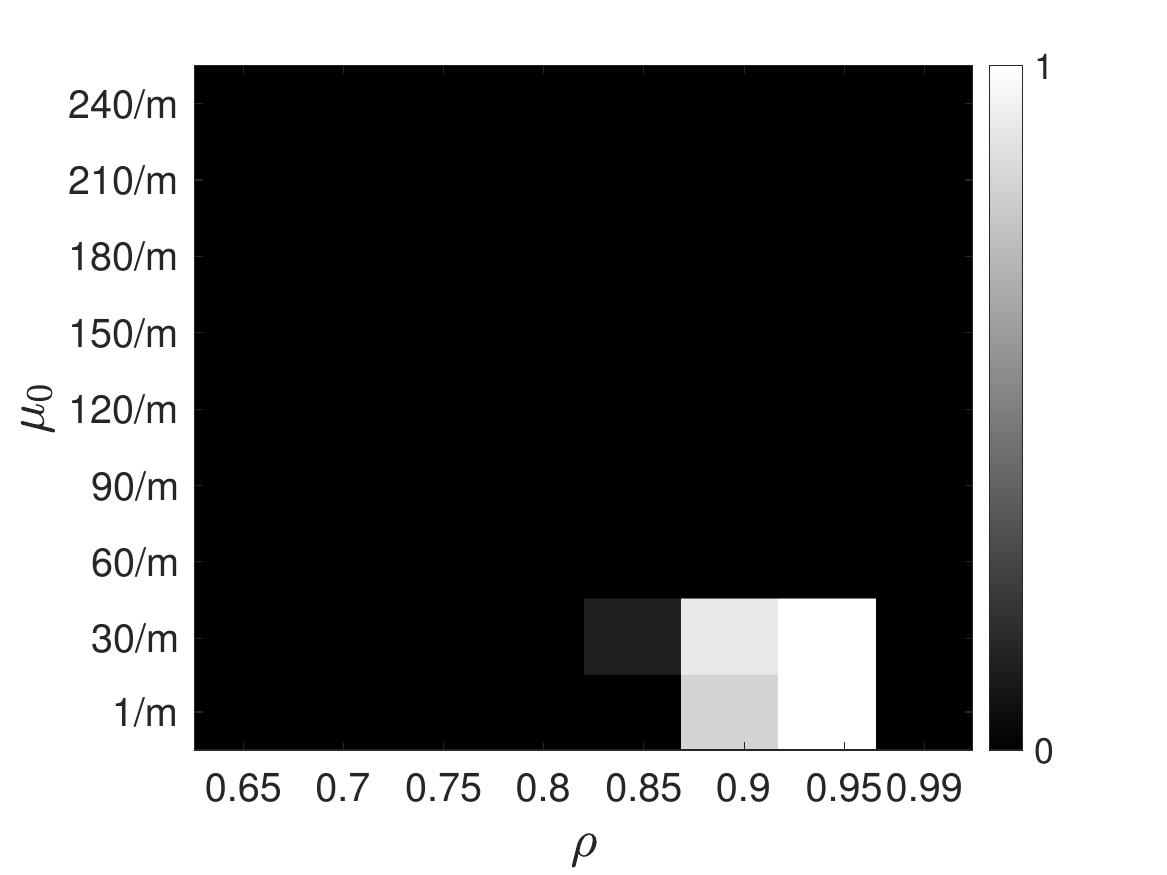}}
		\caption{SGD \label{fig:RMS SGD}}
	\end{subfigure}
	\begin{subfigure}{0.32\linewidth}
		\centerline{
			\includegraphics[width=1\textwidth]{fig/RMS_IGD_algparam}}
		\caption{ISG \label{fig:RMS IGD 3}}
	\end{subfigure}
	\begin{subfigure}{0.32\linewidth}
		\centerline{
			\includegraphics[width=1\textwidth]{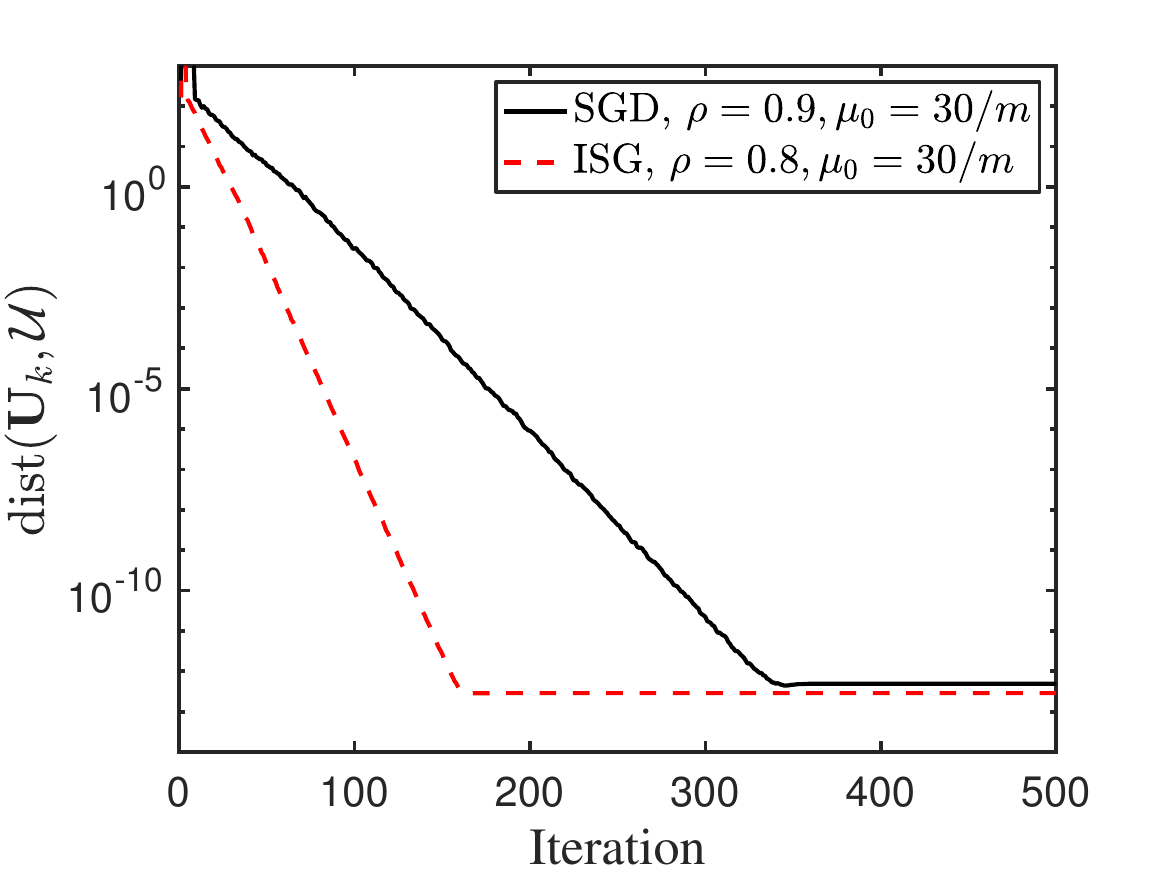}}
		\caption{ISG vs. SGD \label{fig:RMS IGD vs SGD}}
	\end{subfigure}
	\vfill
	\begin{subfigure}{0.32\linewidth}
		\centerline{
			\includegraphics[width=1\textwidth]{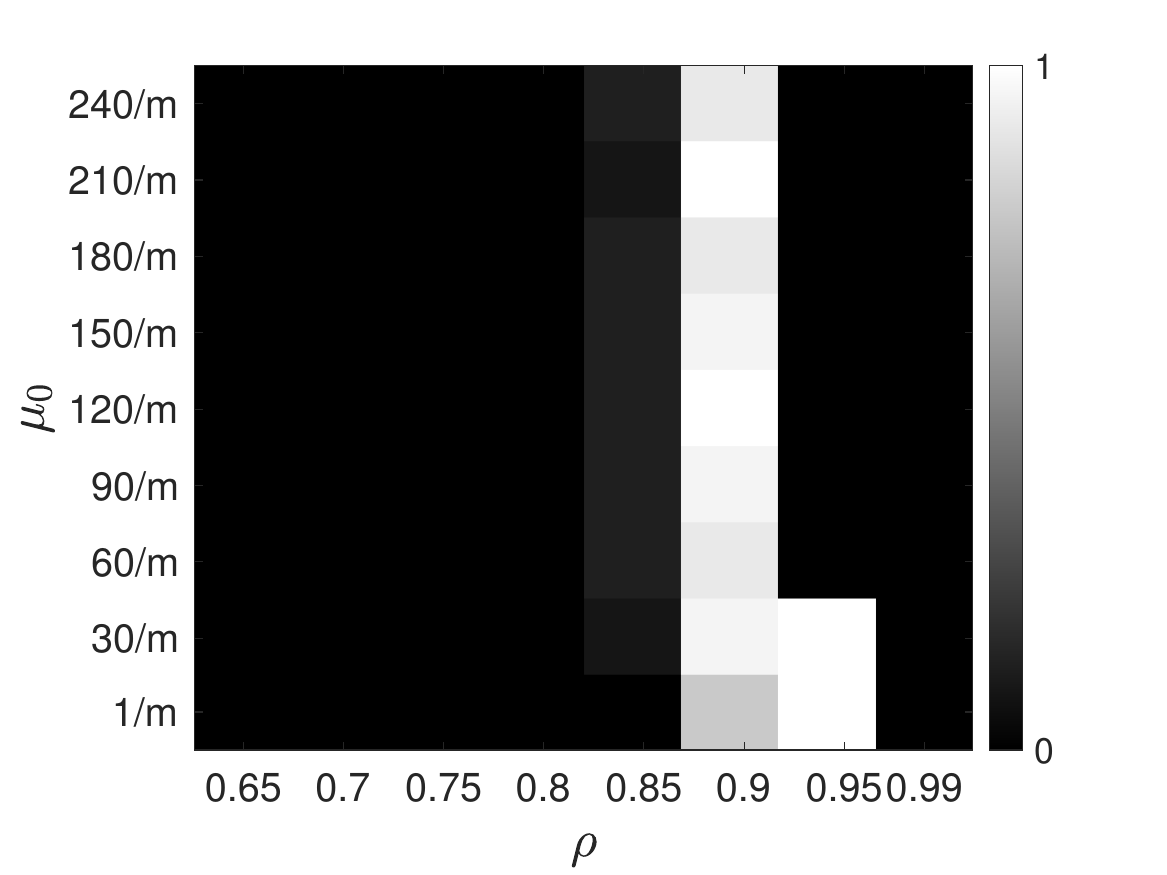}}
		\caption{SPL \label{fig:RMS SPL}}
	\end{subfigure}
	\begin{subfigure}{0.32\linewidth}
		\centerline{
			\includegraphics[width=1\textwidth]{fig/RMS_IPL_algparam}}
		\caption{IPL \label{fig:RMS IPL 2}}
	\end{subfigure}
	\begin{subfigure}{0.32\linewidth}
		\centerline{
			\includegraphics[width=1\textwidth]{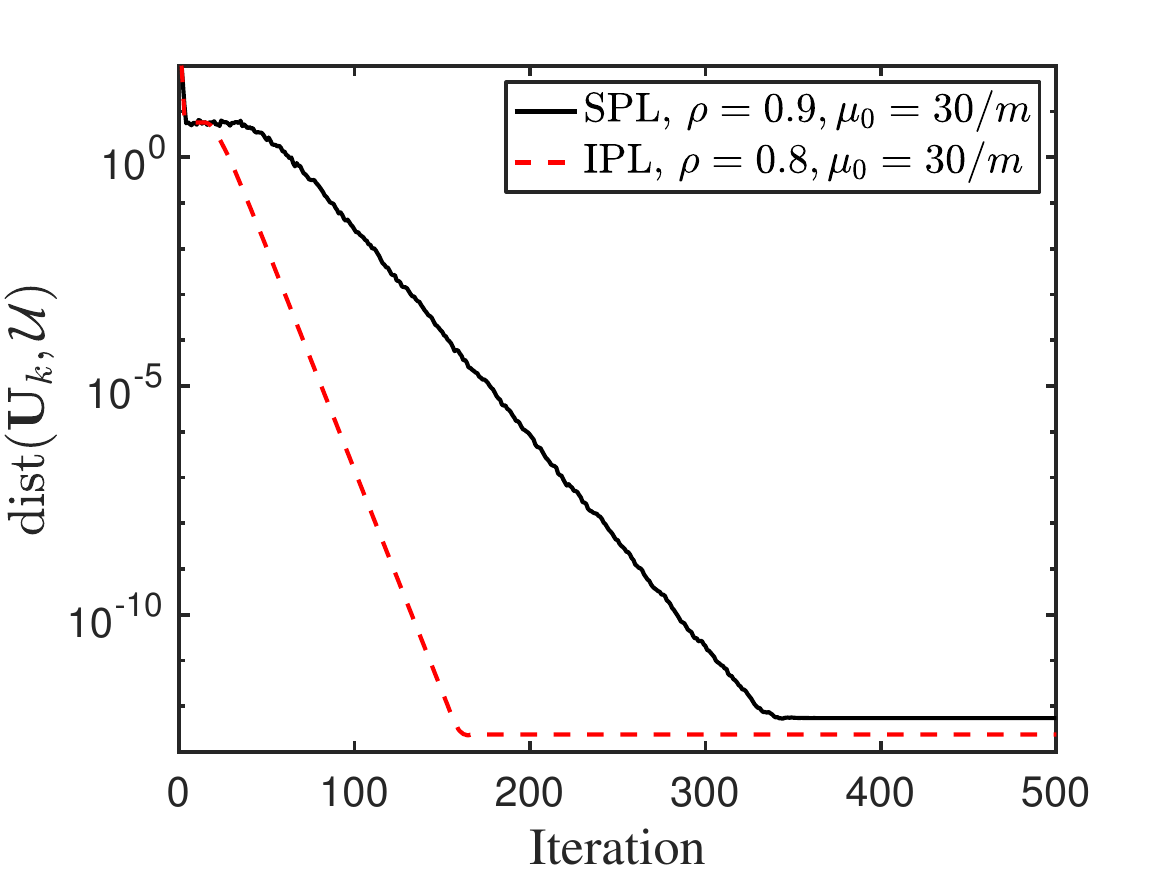}}
		\caption{IPL vs. SPL \label{fig:RMS IPL vs SPL}}
	\end{subfigure}
	\caption{\small Comparison between incremental methods and SGM and stochastic methods for solving the RMS  problem.  }\label{fig:incre vs SGM sto}
\end{figure}

We end this subsection by commenting on IPL. It can be observed from  \Cref{fig:RMS IPL,fig:RMS IPP} that IPL and IPP have nearly the same performance for different choices of the pair of stepsize parameters $(\rho,\mu_0)$.   However, IPP requires an inner solver for  its subproblem \eqref{eq:incremental proximal point},  which renders it computationally inefficient.  By contrast, there exists a closed-form solution for the subproblem \eqref{eq:incremental prox-linear} of  IPL  when $h_i$ in \eqref{eq:local linearization} is the absolute value function, which is the case for the RMS problem. This  implies that the update of IPL is as efficient as that of ISG and much more efficient than the update  of IPP. Recall that IPL is much more robust than ISG with respect to the choice of the initial stepsize $\mu_0$. Thus, we recommend using IPL whenever its subproblem has a closed-form solution. In the sequel, we will omit the results of IPP as its performance is almost the same as IPL.

\subsection{Comparison with SGM and stochastic methods}
In this subsection, we compare the incremental methods with  SGM and stochastic methods.  We apply the same geometrically diminishing stepsizes  to both SGM and stochastic methods. All the  experimental settings remain the same as those in the last subsection.   The result is displayed in  \Cref{fig:incre vs SGM sto}.   

 One can observe from \Cref{fig:RMS SGM,fig:RMS IGD 2} that ISG can tolerate a much smaller choice of $\rho$ than SGM. The smallest workable $\rho$ for SGM is 0.93, while ISG can set $\rho = 0.8$; see \Cref{fig:RMS IGD vs SGM} for a  comparison of the convergence speeds of ISG and SGM when their smallest workable decay factors $\rho$ are used. 

In the comparison with stochastic methods,  if a stochastic method runs for $T$ iterations,  we actually count the iteration number  as ${T}/{m}$.  As observed from \Cref{fig:RMS SGD,fig:RMS IGD 3} and \Cref{fig:RMS SPL,fig:RMS IPL 2}, incremental methods can  choose a smaller $\rho$ than their stochastic counterparts. In particular, incremental methods work with $\rho = 0.8$, while the smallest workable $\rho$ for the  stochastic methods is 0.9. \Cref{fig:RMS IGD vs SGD,fig:RMS IPL vs SPL} compare the convergence speeds of incremental and stochastic methods using the smallest possible $\rho$ of each algorithm.  The superiority of the incremental methods is clear.

\section{Conclusion}
In this work, we introduced a family of incremental methods for solving the weakly convex minimization problem \eqref{eq:problem} and developed a unified framework for analyzing their convergence rates. In particular, we showed that these incremental methods have an iteration complexity of $\calO(\varepsilon^{-4})$ for computing an $\varepsilon$-nearly stationary point of problem \eqref{eq:problem}, and that when applied to a sharp instance of problem~\eqref{eq:problem} using a good initialization and geometrically diminishing stepsizes, they converge linearly to an optimal solution. Our work is the first to extend the convergence rate analysis of incremental methods from the nonsmooth convex regime to the weakly convex regime, which covers a large class of nonsmooth nonconvex problems. We also conducted a series of numerical experiments on the RMS problem~\eqref{eq:rms factorization} to demonstrate the efficacy of the incremental methods. We found that the incremental subgradient method has the simplest implementation, while the incremental proximal point and prox-linear methods are more robust with respect to the choice of the initial stepsize. Moreover, our numerical results suggested that the incremental prox-linear method would be the method of choice when its subproblem can be solved analytically, which is the case for the two motivating applications in \Cref{sec:mov}.  Our experiments clearly showed the superiority of incremental methods over their stochastic counterparts or the full subgradient method, which explains in part why incremental methods are so widely used in practice.

\bibliography{nonconvex}
\bibliographystyle{siamplain}

\end{document}